\documentclass[11pt, reqno]{amsart}

\usepackage{amsmath,amssymb,amsfonts,amsthm}
\usepackage{mathrsfs}
\usepackage{enumitem}
\usepackage[usenames]{color}
\usepackage{hyperref}

\allowdisplaybreaks
\numberwithin{equation}{section}

\theoremstyle{definition}
\newtheorem{definition}{Definition}[section]

\theoremstyle{remark}
\newtheorem{remark}[definition]{Remark}

\theoremstyle{plain}
\newtheorem{proposition}[definition]{Proposition}
\newtheorem{theorem}[definition]{Theorem}
\newtheorem{lemma}[definition]{Lemma}
\newtheorem{result}[definition]{Result}
\newtheorem{corollary}[definition]{Corollary}

\definecolor{Wine}{rgb}{0.5,0,0.05}
\definecolor{DPurple}{rgb}{0.46,0.2,0.69}
\definecolor{Red}{rgb}{1,0,0}

\newcommand{\zbar}{\overline{z}}
\newcommand{\wbar}{\overline{w}}

\newcommand*{\defeq}{\mathrel{\vcenter{\baselineskip0.5ex \lineskiplimit0pt \hbox{\scriptsize.}\hbox{\scriptsize.}}}=}
\newcommand*{\defines}{=\mathrel{\vcenter{\baselineskip0.5ex \lineskiplimit0pt \hbox{\scriptsize.}\hbox{\scriptsize.}}}}

\newcommand\pdshort[2]{\partial^{{#2}}_{{#1}}}
\newcommand{\levi}{\mathscr{L}}

\newcommand{\OM}{\Omega}
\newcommand{\unitdisk}{\mathbb{D}}
\newcommand{\cspdmaltalt}{\mathcal{Q}^{\alpha,a,h}}
\newcommand{\cspdmalt}{\mathcal{Q}^{1/(p-1),a,h}}
\newcommand{\bldblkdm}{\mathcal{R}}

\newcommand{\smoo}{\mathcal{C}}
\newcommand{\hol}{\mathcal{O}}

\newcommand{\uni}{\mathsf{U}}

\newcommand{\koba}{\mathsf{k}}
\newcommand{\dkoba}{\kappa}

\newcommand{\hypsrf}{\mathscr{H}}

\newcommand{\mani}{\mathcal{M}}
\newcommand{\xpoint}{\mathfrak{z}^{\raisebox{-2pt}{$\scriptstyle {x}$}}}

\newcommand{\C}{\mathbb{C}} 
\newcommand{\R}{\mathbb{R}}
\newcommand{\Z}{\mathbb{Z}}
\newcommand{\N}{\mathbb{N}}
\newcommand{\posint}{\mathbb{Z}_{+}}

\newcommand{\inv}{\mathsf{inv}}
\newcommand{\grph}{\mathsf{gr}}

\newcommand{\rprt}{\mathsf{Re}}
\newcommand{\iprt}{\mathsf{Im}}
\newcommand{\distance}{\mathrm{dist}}
\newcommand{\dtb}[1]{\delta_{#1}}
\newcommand{\bcdot}{\boldsymbol{\cdot}}
\newcommand{\bdy}{\partial}
\newcommand{\range}{\mathsf{range}}
\newcommand{\domain}{\mathsf{dom}}

\begin{document}

\title[A weak notion of visibility]{A weak notion of visibility, a family of 
\\ examples, and Wolff--Denjoy theorems}

\author{Gautam Bharali}
\address{Department of Mathematics, Indian Institute of Science, Bangalore 560012, India}
\email{bharali@iisc.ac.in}

\author{Anwoy Maitra}
\address{Department of Mathematics, Indian Institute of Science, Bangalore 560012, India}
\email{anwoymaitra@iisc.ac.in}

\begin{abstract}
We investigate a form of visibility introduced recently by Bharali and
Zimmer\,---\,and shown to be possessed by a class of domains called Goldilocks domains.
The range of theorems established for these domains stem from this form of visibility together
with certain quantitative estimates that define Goldilocks domains. We show that some
of the theorems alluded to follow \emph{merely} from the latter notion of
visibility. We call those domains that possess this property visibility domains with respect to
the Kobayashi distance. We provide a sufficient condition for a domain in $\C^n$ to be a
visibility domain. A part of this paper is devoted to constructing a family of domains
that are visibility domains with respect to the Kobayashi distance but are \emph{not}
Goldilocks domains. Our notion of visibility is reminiscent of uniform visibility
in the context of CAT(0) spaces. However, this is an imperfect analogy because, given a
bounded domain $\OM$ in $\C^n$, $n\geqslant 2$, it is, in general, not even known whether 
the metric space $(\OM, \koba_{\OM})$ (where $\koba_{\OM}$ is the Kobayashi distance)
is a geodesic space. Yet, with just this weak property, we establish two
new Wolff--Denjoy-type theorems. 
\end{abstract}

\keywords{Kobayashi distance, Kobayashi metric, taut domains, visibility, Wolff--Denjoy theorem}
\subjclass[2010]{Primary: 32F45, 32H50, 53C23; Secondary: 32U05}

\maketitle

\vspace{-1.0cm}
\section{Introduction and statement of main results}\label{S:intro}

This work is motivated by the results\,---\,ranging from the boundary behaviour of complex geodesics
to the dynamics of iterations of holomorphic maps\,---\,in a recent work by Bharali and Zimmer
\cite{Bharali_Zimmer}. In that work, the authors introduce a class of bounded domains in $\C^n$, called
Goldilocks domains, and establish for these domains the range of results alluded to. Given any bounded
domain $\OM\subset \C^n$, let $\koba_{\OM}$ be the Kobayashi distance on $\OM$ and
$\dkoba_{\OM} : \OM\times \C^n\cong T^{1,0}\OM\to [0, +\infty)$ be the infinitesimal Kobayashi
metric (also called the Kobayashi--Royden metric).
Goldilocks domains are defined in terms of certain quantitative bounds from below on
$\dkoba_{\OM}(z;\,\bcdot)$ and from above on $\koba_{\OM}(o, z)$ (where $o$ is some chosen
point in $\OM$) as $z\to \bdy{\OM}$\,---\,see subsection~\ref{SS:intro_caltrops} below for a precise
definition. The results in \cite{Bharali_Zimmer} are a consequences of these bounds. In proving some
of the major results in \cite{Bharali_Zimmer}, these bounds play two separate roles:
\begin{enumerate}[leftmargin=22pt]
  \item[$(a)$] in controlling the oscillation of holomorphic maps, the magnitudes of their derivatives, etc.,
  along sequences approaching $\bdy{\OM}$; and
  
  \item[$(b)$] in establishing that $(\OM, \koba_{\OM})$ has certain consequential features\,---\,first
  identified by Eberlein and O'Neill\,---\,possessed by manifolds with negative sectional curvature.
  
\end{enumerate}
The property hinted at by $(b)$ is a purely geometric (i.e., not quantitative) property
reminiscent of visiblity in the sense of
Eberlein--O'Neill \cite{Eberlein_ONeill}. It is used in a fundamental way in the above-mentioned results.
So, it is natural to ask whether the conclusions of those results would hold true in domains that \textbf{merely}
have the geometric property alluded to\,---\,i.e., without assuming the quantitative estimates that define
Goldilocks domains. We coin a term for those domains that have this property via the following definition
(see subsection~\ref{SS:intro_visibility} for the definition of a $(\lambda, \kappa)$-almost-geodesic):

\begin{definition}\label{D:visibility_domain}
	Let $\OM$ be a bounded domain in $\C^n$. We say that $\OM$ is a \emph{visibility domain
	with respect to the Kobayashi distance} (or just \emph{visibility domain} for brevity)
	if, given any $\lambda \geqslant 1$ and
	$\kappa \geqslant 0$, for each pair of distinct points $\xi, \eta \in \bdy\OM$ and each pair
	of $\overline{\OM}$-open neighbourhoods $V$ and $W$ of $\xi$ and $\eta$,
	respectively, such that $\overline{V} \cap \overline{W} = \varnothing$, there exists a compact
	subset $K$ of $\OM$ such that the image of each
	$(\lambda,\kappa)$-almost-geodesic $\sigma : [0,L] \to \OM$ with $\sigma(0) \in V$
	and $\sigma(L) \in W$ intersects $K$.
\end{definition}

While the notion in the above definition is strongly reminiscent of the notion of visibility
manifolds\,---\,especially
in view of \cite[pp.~54--55]{Ballmann_Gromov_Schroeder}\,---\,we must point out that the analogy
is imperfect. For instance, given a
bounded domain in $\C^n$, $n\geqslant 2$, it is, in general, not even known whether 
the metric space $(\OM, \koba_{\OM})$ is a geodesic space. It is for this reason that
Definition~\ref{D:visibility_domain} features $(\lambda,\kappa)$-almost-geodesics, which
serve as substitutes for geodesics.

\smallskip

One might ask: is there a reasonably rich collection of domains that are visibility domains
with respect to the Kobayashi distance? The answer to this is, ``Yes,'' since any Goldilocks
domain is a visibility domain with respect to the Kobayashi distance, and\,---\,as shown in
\cite{Bharali_Zimmer}\,---\,the Goldilocks property admits a very wide range of domains.
However, Definition~\ref{D:visibility_domain} would be interesting only if one knew that there
exist visibility domains that are \textbf{not} Goldilocks
domains. A major part of this paper is devoted to showing that there is a rich family of
domains of this sort. We introduce these domains in subsection~\ref{SS:intro_caltrops}.
In other words, Definition~\ref{D:visibility_domain} is not just a geometrization of the Goldilocks property
but also admits domains in $\C^n$ that are fundamentally different from Goldilocks domains.

\smallskip

Most consequences of visibility in the sense of \cite{Eberlein_ONeill} have been extended
to CAT(0) spaces\,---\,see \cite[Chapter~II]{Bridson_Haefliger}, for instance. Uniform visibility
is the analogue, in the context of CAT(0) spaces, of the property given in Definition~\ref{D:visibility_domain}.
Now, a proper CAT(0) space is uniformly visible if and only if it is Gromov hyperbolic.
There is a reason for mentioning this: many statements that one would like to prove for the metric
space $(\OM, \koba_{\OM})$ would follow very easily if this space were Gromov hyperbolic. However,
Gromov hyperbolicity is a property that is \emph{extremely} difficult to establish for $\koba_{\OM}$
for $\OM\subset \C^n$ when $n\geqslant 2$\,---\,see  \cite{Balogh_Bonk, Zimmer2016} for some positive
instances. Visibility, in the sense of Definition~\ref{D:visibility_domain}, is much easier to show.
In Theorem~\ref{T:gen_visibility-lemma} below we present fairly mild conditions for a
bounded domain in $\C^n$ to be a visibility domain.
It is this theorem that we use to show that the domains introduced in
subsection~\ref{SS:intro_caltrops} are visibility domains with respect to the Kobayashi
distance. We expect that their construction would serve as a general recipe for constructing
visibility domains.

\smallskip

Returning to the question in our first paragraph: the link between Gromov hyperbolicity and
the property in Definition~\ref{D:visibility_domain}, via analogies to uniform visibility,
continues to motivate (as in the case of \cite{Bharali_Zimmer}) certain key moves
in proving analogues of some of the results in \cite{Bharali_Zimmer}. But we show
here that the roles of the quantitative bounds (which also define Goldilocks
domains) identified in $(a)$ above can often be managed by the visibility property alone.
This is the content of our results in Section~\ref{S:props_visibility},
which may be of independent interest. With these inputs, we can systematically approach
several applications for which visibility is well-suited\,---\,some of which will be a part of forthcoming
work. In this paper, we prove two
Wolff--Denjoy-type theorems, which we introduce in
subsection~\ref{SS:intro_Wolff_Denjoy}.

\smallskip

We now introduce the main theorems of this paper.

\subsection{Visibility domains that are not Goldilocks domains}\label{SS:intro_caltrops}
We begin with the definition of a Goldilocks domain. For this, we shall need two quantities.
Given a bounded domain $\OM$ and a point $z\in \OM$, $\delta_{\OM}(z)$ will denote the (Euclidean)
distance from $z$ to $\C^n\setminus\OM$. Next, we define:
\[
  M_{\OM}(r) \defeq \sup\Big\{ \frac{1}{\dkoba_{\OM}(z;v)} \mid \delta_{\OM}(z) \leqslant r
  \text{ and } \|v\| =1\Big\},
\]
where $\|\bcdot\|$ denotes the Euclidean norm (the choice of a norm is actually irrelevant to the purpose
that $M_{\OM}$ serves). From the definition of $\dkoba_{\OM}$, it is easy to see that $M_{\OM}$
expresses the lower bound for $\dkoba_{\OM}$ on the unit sphere in $T^{1,0}\OM$ in
terms of the distance from $\C^n\setminus \OM$. 

\begin{definition}\label{D:Goldilocks_domain}
A bounded domain $\OM \subset \C^n$ is called a \emph{Goldilocks domain} if
\begin{enumerate}
	\item for some (hence any) $\epsilon >0$ we have
	\[
	  \int_0^\epsilon \frac{1}{r} M_\Omega\left(r\right) dr < \infty, \text{ and}
	\]
	\item for each $z_0 \in \OM$ there exist constants $C, \alpha > 0$ (that depend on $z_0$) such that 
	\begin{equation}\label{E:Goldilocks_koba-distance_bound}
	  \koba_{\OM}(z_0, z) \leqslant C + \alpha \log\frac{1}{\delta_{\OM}(z)} \quad \forall z\in \OM.
	\end{equation}
\end{enumerate}
\end{definition}
The quantitative bounds in the above definition encode the following idea: in a Goldilocks domain, 
$\dkoba_{\OM}(z;\,\bcdot)$ cannot grow too slowly and 
$\koba_{\OM}(z_0, z)$ cannot grow too rapidly as $z\to \bdy{\OM}$ (this is the rationale for
the term ``Goldilocks domains'').

\smallskip

The latter has the following geometric implication:  if $\OM$ is a Goldilocks domain, then 
$\bdy{\OM}$ can neither have outward-pointing cusps nor points at which $\bdy{\OM}$ is flat to infinite order and is,
in a precise sense, too flat. One may intuit the assertion about outward-pointing cusps with just a little work:
a classical argument for planar domains reveals that Condition~2 above fails for such domains. This
is the intuition behind a family of domains\,---\,which we call \emph{caltrops}\,---\,that are \textbf{not} Goldilocks
domains, but whose geometry is sufficiently well-behaved that it is reasonable to expect them to be
visibility domains. With this, we make the following

\begin{definition}\label{D:caltrop}
	A bounded domain $\OM\subset \C^n$, $n \geqslant 2$, is called a \emph{caltrop} if 	there exists a finite
	set of exceptional points $\{q_1,\dots, q_N\}\subset \bdy{\OM}$ such that 
	$\bdy{\OM}\setminus \{q_1,\dots, q_N\}$ is $\smoo^2$-smooth, if $\bdy{\OM}$
	is strongly Levi-pseudoconvex at each point in $\bdy{\OM}\setminus \{q_1,\dots, q_N\}$,
	and if for each exceptional point $q_j$, $j = 1,\dots N$, there exists a connected open neighbourhood
	$V_j \ni q_j$ such that $\OM\cap V_j$ is described as follows:
	there exist constants $p_j \in (1, 3/2)$ and $C_j > 1$,
	a unitary transformation $\uni^{(j)}$,
	and a continuous function $\psi_j : [0, A_j] \to [0, +\infty)$ (where $A_j > 0$)
	with the properties mentioned below such that
	$\uni_j(\OM\cap V_j)$ is a ``solid of revolution'' given by
	\[
	  \uni_j(\OM\cap V_j) = 
	  \big\{ (z_1,\dots, z_n)\in \C^n \mid \rprt(z_n) \in (0, A_j), \;
	  \iprt(z_n)^2 + \sum\nolimits_{1\leqslant j\leqslant (n-1)}\!|z_j |^2
	  < \psi_j\big( \rprt(z_n) \big)^2 \big\},
	\]
	where we write $\uni_j \defeq \uni^{(j)}(\bcdot - q_j)$. Each function $\psi_j$ has the following
	properties:
	\begin{itemize}
		\item $\psi_j$ is of class $\smoo^2$ on $(0, A_j)$;
		\item for each $x\in [0, A_j]$, we have
		\[
		  (1/C_j)\,x^{p_j}\,\leqslant \psi_j(x)\,\leqslant C_j\,x^{p_j};
		\]
		\item $\psi_j$ is strictly increasing and $\psi'_j$ is increasing on $(0, A_j)$; and
		\item $\lim_{x\to 0^+}\psi_j(x)\psi''_j(x) = 0$.	
	\end{itemize}
\end{definition}

A few words about the functions $\psi_j$ in the above definition\,---\,and about the last
(somewhat technical-looking) property\,---\,are in order. These
functions are meant to quantify the fact that, around each point of $\bdy{\OM}$ at which it is
non-smooth, the boundary resembles the following real (singular) hypersurface
\[
  \big\{ (z_1,\dots, z_n)\in \C^n \mid \rprt(z_n) \in (0, A), \;
	  \iprt(z_n)^2 + \sum\nolimits_{1\leqslant j\leqslant (n-1)}\!|z_j |^2
	  = \rprt(z_n)^{2p} \big\}
\]
for some $p\in (1, 3/2)$. The latter models a H{\"o}lderian cusp that is not too sharp.

\smallskip

The reader may wonder, given that several specific properties must hold true simultaneously in
a caltrop, whether such a domain as described in Definition~\ref{D:caltrop} can even exist. We show
in Section~\ref{S:caltrops} that caltrops do exist. With this,  the assertion about caltrops that is of
greatest interest to us is:

\begin{theorem}\label{T:visibility-caltrops}
	Caltrops are visibility domains with respect to the Kobayashi distance. However, a caltrop is \emph{not}
	a Goldilocks domain.
\end{theorem}

The proof of this theorem requires 
several supporting results\,---\,about which we shall say more presently\,---\,plus
a sufficient condition for a bounded domain to be a visibility domain with respect to the Kobayashi distance.
We discuss this sufficient condition next.

\subsection{A sufficient condition for visibility}\label{SS:intro_visibility}
The following is the sufficient condition that we have alluded to several times in this section. 
	
\begin{theorem}[\textsc{General Visibility Lemma}]\label{T:gen_visibility-lemma}
	Let $\OM \subset \C^n$ be a bounded domain.
	Suppose there exists a $\smoo^1$-smooth strictly increasing
	function $f : (0,+\infty) \to \R$ such that
	\begin{itemize}
		\item $f(t)\to +\infty$ as $t\to +\infty$; and
		\item for some $z_0$, we have
		\[
		  \koba_{\OM}(z_0, z) \leqslant f\Big( \frac{1}{\dtb{\OM}(z)} \Big) \quad \forall z \in \OM.
		\]
	\end{itemize}
	Assume that $M_{\OM}(t) \to 0$ as $t \to 0$ and that there exists an $r_0 > 0$ such that
	\begin{equation} \label{E:intgr_cndn_M}
	  \int_0^{r_0} \frac{M_{\OM}(r)}{r^2}\,f'\!\!\left( \frac{1}{r} \right) dr < \infty.
	\end{equation}
	Then, $\OM$ is a visibility domain with respect to the Kobayashi distance.
\end{theorem}
It is clear\,---\,comparing Theorem~\ref{T:gen_visibility-lemma} with Conditions~(1) and (2) in
Definition~\ref{D:Goldilocks_domain}\,---\,that our result is influenced by the definition
of Goldilocks domains. Among our motivations were:
\begin{itemize}
	\item that our conditions account for the estimates on $\koba_{\OM}$ and $\dkoba_{\OM}$
	when $\OM$ is any of the planar domains referred to in subsection~\ref{SS:intro_caltrops} with
	$\bdy\OM$ having outward-pointing cusps (which also play a central role in establishing that caltrops
	are visibility domains); and
	
	\item that elements of the proof of the main visibility result in \cite{Bharali_Zimmer}, namely:
	\cite[Theorem~1.4]{Bharali_Zimmer}, continue to be useful in establishing visibility (in the
	sense of Definition~\ref{D:visibility_domain}).
\end{itemize} 
Observe that the inequalities that define Goldilocks
domains are subsumed by Theorem~\ref{T:gen_visibility-lemma}: for these domains, just set
\[
  f(t) = C + \alpha\log(t), \quad t\in (0, +\infty),
\]
with $C, \alpha>0$ as in 
\eqref{E:Goldilocks_koba-distance_bound},
in the latter theorem. The proof of Theorem~\ref{T:gen_visibility-lemma} is given in
Section~\ref{S:gen_visibility-lemma}.

\smallskip

There are two essential matters relating to visibility domains that we had deferred. We address them here.
First, we give a definition for $(\lambda, \kappa)$-almost-geodesics:

\begin{definition}[Bharali--Zimmer, \cite{Bharali_Zimmer}]\label{D:lambda_kappa}
Let $\OM \subset \C^n$ be a bounded domain and $I \subset \R$ an interval.
For $\lambda \geqslant 1$ and $\kappa \geqslant 0$, a curve $\sigma:I \to \OM$ is called a
\emph{$(\lambda, \kappa)$-almost-geodesic} if 
\begin{enumerate}[leftmargin=22pt] 
	\item\label{item:quasi} for all $s,t \in I$,  
	\[
	  \lambda^{-1} |t-s| - \kappa \leqslant \koba_{\OM}(\sigma(s), \sigma(t))
	  \leqslant \lambda |t-s| +  \kappa; \text{ and}
	\]
	\item $\sigma$ is absolutely continuous (whence $\sigma'(t)$ exists for almost every $t\in I$)
	and, for almost every $t \in I$, $\dkoba_{\OM}(\sigma(t); \sigma'(t)) \leqslant \lambda$.
\end{enumerate}
\end{definition}

Secondly, the discussion surrounding visibility domains suggests that, given any pair of points in $\OM$,
there exists (for any $\lambda\geqslant 1$ and $\kappa > 0$) a $(\lambda, \kappa)$-almost-geodesic
joining them, but why must this be so? In fact, this is true for \textbf{any} bounded domain 
$\OM$\,---\,as shown by \cite[Proposition~4.4]{Bharali_Zimmer}.

\smallskip

The General Visibility Lemma (i.e., Theorem~\ref{T:gen_visibility-lemma}) is our key tool for showing that
caltrops are visibility domains.  We must derive appropriate upper bounds for
$\koba_{\OM}(o, z)$ (where $o$ is a chosen point in $\OM$) and lower
bounds for $\dkoba_{\OM}(z;\,\bcdot)$ for $z$ in the caltrop $\OM$, $z$ sufficiently close to
$\bdy\OM$. The following points describe
very briefly the challenging parts of the proof of Theorem~\ref{T:visibility-caltrops}:
\begin{enumerate}[leftmargin=22pt]
  \item[$(i)$] We explicitly calculate the Kobayashi distance on a model domain $D\Subset \C$,
  $\bdy{D}$ having an outward-pointing cusp, which is carefully chosen keeping in view the geometry
  of $\bdy{\OM}\cap V_j$, $j = 1,\dots, N$,
  the latter being an outward-pointing cusp of $\bdy\OM$  as described in
  Definition~\ref{D:caltrop}. Then, $\C$-affine embeddings of $D$ into $\OM\cap V_j$
  allow us to estimate $\koba_{\OM}$ using the fact that these embeddings are
  contractive relative to the Kobayashi distance.  
  \item[$(ii)$] We introduce a trick of estimating the Sibony 
  pseudometric \cite[Proposition~6]{Sibony} for $\OM$ in $\OM\cap V_j$,
  $j = 1,\dots, N$.
  The relationship between the latter pseudometric and $\dkoba_{\OM}$ leads to an 
  estimate for $\dkoba_{\OM}(z;\,\bcdot)$ from below for
  $z\in \OM\cap V_j$.
\end{enumerate}
The argument summarized by $(i)$ above requires several results, which are presented in
Section~\ref{S:comparison_dom}. The proof of Theorem~\ref{T:visibility-caltrops}
is given in Section~\ref{S:visibility-caltrops}. We expect the procedures in this
proof to serve as a general recipe for constructing new visibility domains.

\subsection{Wolff--Denjoy theorems for visibility domains}\label{SS:intro_Wolff_Denjoy}
The classical Wolff--Denjoy theorem is as follows (in this paper, 
$\unitdisk$ will denote the open unit disk in $\C$ with centre $0$):

\begin{result}[Denjoy\cite{Denjoy}, Wolff \cite{Wolff}]\label{R:WD}
Suppose $F:\unitdisk \to \unitdisk$ is a holomorphic map. Then, either:
\begin{enumerate}
	\item $F$ has a fixed point in $\unitdisk$; or
	\item there exists a point $\xi \in \partial \unitdisk$ such that $\lim_{\nu\to \infty} F^{\nu}(z) = \xi$
	for every $z \in \unitdisk$, this convergence being uniform on compact subsets of $\unitdisk$.
\end{enumerate}
\end{result}

There has been sustained interest in understanding the behaviour of iterates of a map
$F: X\to X$ where $X$ is a space that\,---\,in some appropriate sense\,---\,resembles $\unitdisk$
while $F$ possesses some degree of regularity that enables a generalization of Result~\ref{R:WD}
to $F : X\to X$. The above result was extended to the unit (Euclidean) ball in $\C^n$,
for all $n\in \posint$, by Herv{\'e} \cite{Herve}. Abate further generalized this in \cite{Abate88}
to strongly convex domains.
Let $X$ be a visibility manifold in the sense
of \cite{Eberlein_ONeill} (as discussed at the beginning of this section). Then,
one can construct a boundary for $X$ ``at infinity'' that serves as the analogue of
$\bdy\unitdisk$. With this set-up for $X$, Beardon \cite{Beardon1990} generalized the above result
to $F : X\to X$ where $F$ is a strict contraction.

\smallskip

For various reasons having to do with their intrinsic geometry, convex domains
predominate among recent generalizations of the Wolff--Denjoy theorem: see, for instance,
\cite{Beardon97, Budzynska, Abate_Raissy, Zimmer2017} and several of the results in \cite{Karlsson}.
Also see the survey article \cite{Reich_Shoikhet} for a discussion of generalizations to balls in
infinite-dimensional vector spaces and the extensive references therein. 
Visibility in the sense of Definition~\ref{D:visibility_domain} is one of the key ingredients
in the proof by Bharali--Zimmer of a generalization \cite[Theorem~1.10]{Bharali_Zimmer}
of Result~\ref{R:WD} to taut Goldilocks domains.
This extends the Wolff--Denjoy phenomenon to a wide range of domains, including
pseudoconvex domains of finite type\,---\,see \cite[Corollary~2.11]{Bharali_Zimmer}\,---\,which
are, in general, neither convex nor biholomorphic to convex domains. In the following result,
we extend the Wolff--Denjoy phenomenon to \textbf{all} visibility domains with respect to
the Kobayashi distance that are taut:   

\begin{theorem}\label{T:Wolff_Denjoy_gen}
	Suppose $\OM \subset \C^n$ is a visibility domain with respect to the Kobayashi distance that
	is taut and that $F: \OM \to \OM$ is a holomorphic map. Then exactly one of the following holds:
	\begin{enumerate}
		\item for each $z \in \OM$, the orbit $\{ F^{\nu}(z) \mid \nu \in \posint \}$ is
		relatively compact in $\OM$; or
		\item there exists a $\xi \in \bdy \OM$ such that $\lim_{\nu\to \infty}F^{\nu}(z) = \xi$ 
		for every $z\in \OM$, this convergence being uniform on compact subsets of $\OM$.
	\end{enumerate}
\end{theorem}

The statement of the above theorem differs from that of \cite[Theorem~1.10]{Bharali_Zimmer}
in the one respect that the dichotomy presented in the above theorem holds true on any taut
visibility domain, and not just on Goldilocks domains. Furthermore, the hypothesis of
Theorem~\ref{T:Wolff_Denjoy_gen} does admit domains that are not Goldilocks domains\,---\,as
the reader will infer from Corollary~\ref{C:Wolff_Denjoy_example} below.

\smallskip

The proof of \cite[Theorem~1.10]{Bharali_Zimmer} is borne by two distinct ideas. The first is
the following heuristic, which is entirely a consequence of visibility: if we assume that there exist two
strictly increasing sequences $(\nu_i)_{i \geqslant 1},
(\mu_j)_{j \geqslant 1}\subset \posint$ with $F^{\nu_i}(o) \to \xi \in \bdy\OM$,
$F^{\mu_j}(o)\to \eta \in \bdy\OM$,
and $\xi \neq \eta$, then we arrive at a contradiction by analysing
$\koba_{\OM}(F^{\nu_i}(o), F^{\mu_j}(o))$. Briefly: with $\nu_i > \mu_j$, an estimate for
$\koba_{\OM}(F^{\nu_i}(o), F^{\mu_j}(o))$ based on the fact that $F$ is contractive
with respect to $\koba_{\OM}$ is incompatible with an estimate based on the
fact that every $(1, \kappa)$-almost-geodesic (with $\kappa > 0$) joining 
$F^{\nu_i}(o)$ to $F^{\mu_j}(o)$ must pass within a fixed distance of $o$. As this is a
consequence of visibility, this heurisitc informs the proof of Theorem~\ref{T:Wolff_Denjoy_gen}
too. \textbf{However,}
visibility alone does not a priori seem to explain the limits $F^{\nu_i}(o) \to \xi$ and
$F^{\mu_j}(o)\to \eta$ mentioned above. That explanation, under the
assumption that $\OM$ is taut, belongs to the realm described by $(a)$ earlier in this section.
It turns out that (still assuming that $\OM$ is taut) visibility alone is enough to justify these
limits. This is the purport of our results in Section~\ref{S:props_visibility}, which play a 
supporting role in the proof of  Theorem~\ref{T:Wolff_Denjoy_gen}, but may also be of
independent interest.

\smallskip

The dichotomy in the behaviour of the iterations in Theorem~\ref{T:Wolff_Denjoy_gen} is not quite
what is given by the classical Wolff--Denjoy theorem (i.e., Result~\ref{R:WD}). Where, among
the Wolff--Denjoy-type results cited above, the dichotomy given by Result~\ref{R:WD} does hold,
it is a consequence of the domains $\OM$ in question being contractible and of $\bdy\OM$ 
satisfying some non-degeneracy condition: some form of strict convexity; or
strong pseudoconvexity, as in \cite{Huang}; etc.
In view of the many examples presented
in \cite[Section~2]{Bharali_Zimmer}, and given Theorem~\ref{T:visibility-caltrops} about
caltrops, the boundaries of taut visibility domains do not generally have the type of
non-degeneracy mentioned above. 
However, with some conditions on
the topology of $\OM$ (as opposed to the geometry of $\bdy\OM$),   
we can use a result of Abate \cite{Abate_ItThCptDivSeq} to obtain a version of
Theorem~\ref{T:Wolff_Denjoy_gen} whose conclusions more closely resemble those of
Result~\ref{R:WD}. 

\begin{theorem}\label{T:Wolff_Denjoy_special}
	Suppose $\OM \subset \C^n$ is a visibility domain with respect to the Kobayashi distance that is taut,
	and that $\OM$ is of finite topological type. Suppose further that
	\[
	  H^j(\OM; \C) = 0 \quad\text{for each odd $j$, $1\leqslant j\leqslant n$}.
	\]
	Let $F: \OM \to \OM$ be a holomorphic map. Then exactly one of the following holds:
	\begin{enumerate}
		\item $F$ has a periodic point in $\OM$; or
		\item there exists a $\xi \in \bdy \OM$ such that $\lim_{\nu\to \infty}F^{\nu}(z) = \xi$
		for every $z\in \OM$, this convergence being uniform on compact subsets of $\OM$.
	\end{enumerate}
	In the first case, each orbit $\{ F^{\nu}(z) \mid \nu \in \posint \}$, $z\in \OM$, is relatively compact in $\OM$.	
\end{theorem}
\noindent{Recall that for $\OM$ to have finite topological type means that the singular homology groups
$H_j(\OM; \Z)$ are of finite rank for all $j\in \N$.}

\smallskip

We point out that the domains to which Theorem~\ref{T:Wolff_Denjoy_special} applies need
\emph{not} be convex or biholomorphic to a convex domain (a fact that will be emphasised
by the corollary below). In this regard, Theorem~\ref{T:Wolff_Denjoy_special} bears relation to
\cite{Huang} by X.~Huang, in which the dichotomy presented in Result~\ref{R:WD} is established
for bounded topologically contractible strongly pseudoconvex domains. Loosely speaking, a version of the
heuristic discussed right after Theorem~\ref{T:Wolff_Denjoy_gen} appears in \cite{Huang}, although the
specifics that make this heuristic work in \cite{Huang} and for our result differ greatly. (Also, the arguments
in \cite{Huang} suggest that the dichotomy presented in Result~\ref{R:WD} would be very hard to
obtain for the domains of the generality that we consider.)

\smallskip

The final result of this subsection is meant to illustrate tangibly the range of domains\,---\,with an emphasis
on domains that need not be convex or biholomorphic to a convex domain, and with boundaries that
aren't even Lipschitz\,---\,to which the Wolff--Denjoy phenomenon extends.

\begin{corollary}\label{C:Wolff_Denjoy_example}
	Suppose $\OM \subset \C^n$ is either a bounded pseudoconvex domain of finite type
	or a caltrop. Suppose further that $\OM$ is of finite topological type and that
	\[
	  H^j(\OM; \C) = 0 \quad\text{for each odd $j$, $1\leqslant j\leqslant n$}.
	\]
	Let $F: \OM \to \OM$ be a holomorphic map. Then exactly one of the following holds:
	\begin{enumerate}
		\item $F$ has a periodic point in $\OM$; or
		\item there exists a $\xi \in \bdy \OM$ such that $\lim_{\nu\to \infty}F^{\nu}(z) = \xi$
		for every $z\in \OM$, this convergence being uniform on compact subsets of $\OM$.
	\end{enumerate}
	In the first case, each orbit $\{ F^{\nu}(z) \mid \nu \in \posint \}$, $z\in \OM$, is relatively compact in $\OM$.	
\end{corollary}

The proofs of all the results of this subsection are presented in Section~\ref{S:Wolff_Denjoy} below.
\medskip

\section{Technical preliminaries}\label{S:prelims_technical}

This section is dedicated to introducing notation that will recur throughout this paper, and
some known results that will play a supporting role
in the proofs presented in the following sections. This section is divided into three parts. We begin
with some notation (some of which has appeared in passing in Section~\ref{S:intro}) that we shall need.

\subsection{Common notations}
We fix the following notation, which we shall frequently need.
\begin{enumerate}
	\item For $v \in\C^n$, $\|v\|$ will denote the Euclidean norm.
	Given points $z, w \in \C^n$, we shall commit a mild abuse of notation by not distinguishing
	between points and tangent vectors, and denote the Euclidean distance between them as
	$\|z - w\|$.
	\item The maps $\pi_j : \C^n\to \C$, $j = 1,\dots, n$, will denote the projection onto the
	$j$-th factor.
	\item $\unitdisk$ will denote the open unit disk in $\C$ with centre at $0$, 
	while $D(a, r)$ will denote the open disk in $\C$ with radius $r > 0$ and centre $a$.
	\item Given an open set $U\subset \C^n$ and a $\smoo^2$-smooth function
	$\rho : U\to \R$, we will denote by $\levi(\rho)(z; v)$ the quadratic form (called the
	\emph{Levi-form} of $\rho$) determined by the complex Hessian of $\rho$ at $z\in U$:
	\[
	  \levi(\rho)(z; v) \defeq
	  \sum_{1\leqslant j,\,k\leqslant n}\pdshort{z_j\zbar_k}{2}\rho(z)v_j\overline{v}_k
	\]
	for each $v\in \C^n$ (equivalently, for each $v\in T^{1,0}_{z}U$).
\end{enumerate}

\subsection{Facts relating to the Kobayashi geometry of domains}
Let $\OM$ be a domain in $\C^n$. We shall assume that the reader is familiar with the Kobayashi
pseudodistance $\koba_{\OM}$ and the Kobayashi--Royden pseudometric $\dkoba_{\OM}$.
The only comment concerning the basics of these objects that we shall make is that $\koba_{\OM}$ and
$\dkoba_{\OM}$ are related as follows:
\[
   \koba_{\OM}(z, w) =
  \inf_{\gamma\in \mathscr{C}(z,w)}\int_0^1\dkoba_{\OM}(\gamma(t); \gamma'(t))\,dt
  \quad\forall z, w\in \OM,
\]
where $\mathscr{C}(z,w)$ is the set of all piecewise $\smoo^1$
paths $\gamma : [0, 1]\to \OM$ satisfying $\gamma(0) = z$ and $\gamma(1) = w$. This is a
result by Royden \cite{Royden}

\smallskip

We shall need the following estimate on $\koba_{\OM}$:

\begin{result}[{\cite[Proposition~3.5-(1)]{Bharali_Zimmer}}]\label{R:koba_facts_misc}
	Let $\OM$ be a bounded domain in $\C^n$. Fix an open ball $\mathbb{B}(\OM)$ with centre
	$0\in \C^n$ that is so large that $\OM\Subset \mathbb{B}(\OM)$. Let
	\[
	  c \defeq \inf\nolimits_{x\in \overline{\OM},\,\|v\|=1}\dkoba_{\mathbb{B}(\OM)}(x; v).
	\]
	Then, $\koba_{\OM}(z, w) \geqslant c\|z - w\|$ for every $z, w\in \OM$.
\end{result}

Tautness is closely tied to metric geometry associated to the Kobayashi pseudodistance. For a domain
$\OM\subset \C^n$, $n\geq 2$, being taut provides additional information about the complex geometry
of $\OM$. We collect a couple of observations of this nature in the following:

\begin{result}\label{R:koba_facts_miscGeom}
	Let $\OM\subset \C^n$ be a taut domain. Then:
	\begin{enumerate}
		\item\label{Cnclsn:taut_domains_cont} \emph{(see, for instance, \cite[Proposition~3.5.13]{Jarnicki_Pflug})}
		$\dkoba_{\OM}$ is continuous on $\OM\times \C^n$.
		
		\item\label{Cnclsn:taut_domains_pcvx} \emph{(\cite[Theorem~F]{Wu})} If $\OM$ is bounded, then it is
		pseudoconvex.
	\end{enumerate}
\end{result}

In order to prove Corollary~\ref{C:Wolff_Denjoy_example}, we would need to prove that caltrops
(we shall show in the next section that caltrops indeed exist) are taut. The following result will be
useful in this proof.

\begin{result}[{\cite[Corollary~2.4]{Forstneric_Rosay}}]\label{R:koba_dist_lower_bd}
	Let $\OM$ be a bounded domain in $\C^n$ whose boundary is of class $\smoo^2$ and
	strongly Levi-pseudoconvex in $\bdy\OM$-neighbourhoods of two distinct points
	$\xi, \eta\in \bdy\OM$. Then, there is a constant  $C > 0$, which depends on
	$\xi$ and $\eta$, and open neighbourhoods $V_{\xi}$ and $V_{\eta}$ in $\C^n$ of
	$\xi$ and $\eta$, respectively such that
	\[
	  \koba_{\OM}(a, b) \geqslant 2^{-1}\log\frac{1}{\delta_{\OM}(a)}
	  + 2^{-1}\log\frac{1}{\delta_{\OM}(b)} - C
	\]
	for each point $a\in \OM\cap V_{\xi}$ and $b\in \OM\cap V_{\eta}$.
\end{result}

The proof of Theorem~\ref{T:visibility-caltrops} will, at a certain stage, require a precise estimate from
below for $\dkoba_{\OM}$\,---\,where $\OM$ is a bounded domain in $\C^n$\,---\,in the vicinity of
a strictly pseudoconvex point in $\bdy\OM$. Such an estimate is provided by a result of 
D.~Ma \cite[Theorem~B]{Ma}.
Before stating the result that we need, we ought to mention that Ma's result is stated for
domains for which the part of the boundary that is strongly Levi-pseudoconvex is
$\smoo^3$-smooth. However, Ma's techniques are still valid \emph{up to a point}
when this regularity condition is weakened to $\smoo^2$-smooth\,---\,the modifications
required, in essence, are to replace all
occurences of $O(\|x\|^3)$ by $o(\|x\|^2)$
in those steps of the argument that invoke Taylor's theorem. Where this does not
suffice, Balogh and Bonk\,---\,in the sketch of their proof of
\cite[Proposition~1.2]{Balogh_Bonk}\,---\,provide the essential modification needed.
While Balogh--Bonk state their estimate for strongly pseudoconvex domains with
$\smoo^2$-smooth boundary, their proof actually involves \emph{local} estimates, which lead to
the inequalities below. With these clarifications, we state:

\begin{result}[paraphrasing {\cite[Theorem~B]{Ma}} and
{\cite[Proposition~1.2]{Balogh_Bonk}}]\label{R:koba_metric_M-BB} 
	Let $\OM$ be a bounded domain in $\C^n$, $n\geqslant 2$. Let $\mani_0$ be a $\bdy{\OM}$-open
	set that is a $\smoo^2$-smooth hypersurface. Assume that $\mani_0$ admits
	a defining function $\phi$ that is of class $\smoo^2$ on some open set containing $\mani_0$ and
	that there exists a small constant $\sigma > 0$ such that $\levi(\phi)(\xi; v) \geqslant \sigma\|v\|^2$
	at each $\xi\in \mani_0$ and for all $v\in \C^n$. Let $\mani_1\varsubsetneq \mani_0$ be
	a compact subset. Then, there exists
	an $\overline{\OM}$-open neighbourhood, say $\mathcal{V}$, of $\mani_1$ and a constant $C > 0$
	such that
	\[
	  \dkoba_{\OM}(z; v) \geqslant (1 - C\delta_{\OM}(z)^{1/2})\,\frac{\sigma\|v\|^2}{\delta_{\OM}(z)^{1/2}}
	\]
	  for every $z\in \mathcal{V}\cap \OM$ and for every $v\in \C^n$.
\end{result}
 
\begin{remark}
In fact, a much more precise estimate is provided by Ma and Balogh--Bonk than the one stated in the
above result. However, in order to state the latter estimate, one would need to provide certain definitions
that would be a digression from the present discussion. The lower bound for $\dkoba_{\OM}$
stated in Result~\ref{R:koba_metric_M-BB} suffices for our purposes.
\end{remark}

The following result by Sibony will also play an important role in the proof of Theorem~\ref{T:visibility-caltrops}:

\begin{result}[paraphrasing {\cite[Proposition~6]{Sibony}}]\label{R:koba_metric_lower}
	Let $\OM$ be a domain in $\C^n$ and let $p \in \OM$. Suppose $u$ is a negative plurisubharmonic
	function that is of class $\smoo^2$ in a neighbourhood of $p$ and assume that
	\[
	  \levi(u)(p; v) \geqslant c\|v\|^2 \quad \forall v\in \C^n,
	\]
	where $c$ is some positive constant. Then, there is a universal constant $\alpha > 0$ such that
	\[
	  \dkoba_{\OM}(p; v) \geqslant \left(\frac{c}{\alpha}\right)^{1/2}\!\frac{\|v\|}{|u(p)|^{1/2}}.
	\]
\end{result}

\begin{remark}
An important part of \cite{Sibony} is the construction of a pseudometric on $T^{1,0}\OM$\,---\,which
is known today as the Sibony pseudometric\,---\,that is dominated by the Kobayashi pseudometric. The
lower bound in Result~\ref{R:koba_metric_lower} is actually a lower bound for the Sibony pseudometric,
from which the lower bound above for $\dkoba_{\OM}(p; \bcdot)$ is obtained.
\end{remark}

In concluding this section, we collect a few
\smallskip

\subsection{Facts relating to length-minimizing curves}
The fundamental fact that we presuppose in this subsection is that if $\OM$ is a bounded domain in
$\C^n$, then for any two points in $\OM$ and for any $\lambda\geqslant 1$ and $\kappa > 0$
there exists a $(\lambda, \kappa)$-almost-geodesic joining these points: this is the content of Proposition~4.4
of \cite{Bharali_Zimmer} by Bharali--Zimmer. With this understanding, we first present:

\begin{result}[{\cite[Proposition~4.3]{Bharali_Zimmer}}]\label{R:Lipschitz}
	Let $\OM$ be a bounded domain in $\C^n$. For any $\lambda \geqslant 1$ there exists a
	$C = C(\lambda)>0$ such that any $(\lambda, \kappa)$-almost-geodesic
	(where $\kappa\geqslant 0$) $\sigma: [a, b]\to \OM$ is $C$-Lipschitz (with respect to the Euclidean distance). 
\end{result}

We shall also need the following simple lemma, whose proof is essentially a single line following from the definition
of $(\lambda, \kappa)$-quasi-geodesics and the triangle inequality. To clarify: a
\emph{$(\lambda, \kappa)$-quasi-geodesic} in $\OM$ is a function $\sigma: I\to \OM$, where $I$
is an interval, satisfying
the property (\ref{item:quasi}) stated in Definition~\ref{D:lambda_kappa}.

\begin{lemma}\label{L:basic_but_important}
	Let $\OM$ be a bounded domain in $\C^n$. If $\sigma: [a,b]\to \OM$ is a $(1,\kappa)$-quasi-geodesic,
	then for all $t \in [a,b]$ we have
	\[
	  \koba_{\OM}(\sigma(a), \sigma(b))\leqslant \koba_{\OM}(\sigma(a),\sigma(t))
	  + \koba_{\OM}(\sigma(t), \sigma(b))
	  \leqslant \koba_{\OM}(\sigma(a),\sigma(b)) + 3\kappa.
	\]
\end{lemma}

\section{Caltrops exist}\label{S:caltrops}
In this section, we shall construct explicit examples of caltrops. To begin with, we will construct with some care a caltrop
whose boundary has one outward-pointing cusp. We shall then abstract features of this construction to describe
briefly the construction of caltrops with any (finite) number of outward-pointing cusps. Our constructions will be
in $\C^2$ but\,---\,as will become clear\,---\,this is only for simplicity of notation.
\smallskip

We shall call the subset $\OM\cap V_{j}$, $j = 1,\dots, N$, where $V_j$ is as in Definition~\ref{D:caltrop},
a \emph{spike}.

\subsection{A caltrop with a single spike}\label{SS:caltrop_1-spike}
Let $A$ and $\beta$ be positive numbers and let $\psi: [-A, \beta]\to [0, +\infty)$ be a continuous function that is of
class $\smoo^2$ on $(-A, \beta)$ such that
\begin{enumerate}
	\item $\psi(t) \defeq (t + A)^p$ for every $t\in [-A, -B]$, and
	\item $\psi(t) \defeq \sqrt{\beta^2 - t^2}$ for every $t\in (0, \beta)$,
\end{enumerate}
where $B \in (0, A)$ and $p \in (1, 3/2)$. We shall consider the following ``solid of revolution'' given by
\[
  \OM \defeq \{(z, w)\in \C^2 : |z|^2 + |\iprt{w}|^2 < C\psi(\rprt{w})^2, \ -A <\rprt(w) < \beta\},
\]
where $C > 0$ is a small constant whose magnitude we shall specify presently. Let us write
\[
  \rho(z, w) \defeq |z|^2 + |\iprt{w}|^2 - C\psi(\rprt{w})^2, \quad
  (z,w) \in \{(z, w)\in \C^2 : -A < \rprt(w) < \beta\}.
\]
It is easy to check that $\rho$ is a $\smoo^2$-smooth defining function for the real hypersurface
$\bdy\OM \cap \{(z, w)\in \C^2 : -A < \rprt(w) < \beta\}$.
\smallskip

We compute:
\begin{align*}
  \pdshort{z\zbar}{2}\rho &\equiv 1, \\
  \pdshort{z\wbar}{2}\rho &= \pdshort{\zbar w}{2}\rho \equiv 0, \\
  \pdshort{w\wbar}{2}\rho(z, w) &= \frac{1}{2}
  	- \frac{C}{2}\left(\psi^{\prime\prime}(\rprt{w})\psi(\rprt{w}) + \psi'(\rprt{w})^2\right),
\end{align*}
wherever $\rho$ is of class $\smoo^2$. In particular, we have
\begin{equation}\label{E:w-wbar_positive1}
  \pdshort{w\wbar}{2}\rho(z, w) - \frac{1}{2} = 
  	- \frac{Cp(2p-1)}{2}(\rprt{w} + A)^{2(p-1)} \nearrow 0 \text{ as $\rprt{w}\searrow -A$.}
\end{equation}
Furthermore, as $\psi$ is of class $\smoo^2$ on $(-A, \beta)$, we can, by choosing $C > 0$ sufficiently small,
ensure that
\begin{equation}\label{E:w-wbar_positive2}
  \pdshort{w\wbar}{2}\rho(z, w) \geqslant \frac{1}{4} \quad \forall w : -A < \rprt{w}\leqslant 0.
\end{equation}
From \eqref{E:w-wbar_positive1} and \eqref{E:w-wbar_positive2}, we conclude that
$\left.\rho\right|_{\{(z, w)\in \C^2 : -A < \rprt{w} < \delta\}}$ is strictly plurisubharmonic for some positive
constant $\delta \ll 1$. In particular, $\bdy\OM$ is strongly Levi-pseudoconvex at each point
on $\bdy\OM \cap \{(z, w)\in \C^2 : -A < \rprt{w} \leqslant 0\}$. Of course, by construction\,---\,by the
condition $(2)$ on $\psi$, to be precise\,---\,$\bdy\OM$ is strongly Levi-pseudoconvex at each point
on $\bdy\OM \cap \{(z, w)\in \C^2 :  \rprt{w} \geqslant 0\}$. The other properties that $\OM$ must have
for it to be a caltrop follow from the condition $(1)$ on $\psi$.

\subsection{A caltrop with many spikes}
A slight modification of the details described in the previous subsection allows us to show the existence of caltrops with
many spikes. To this end, let us fix constants $A_1,\dots A_N > 1$ and consider a collection of continuous
functions $\psi_j : [-A_j, \beta_j]\to [0, +\infty)$\,---\,where each $\beta_j$ is a constant with
$\beta_j > -1$\,---\,that are of class $\smoo^2$ on $(-A_j, \beta_j)$, such that
\[
  \psi_j(t) \defeq (t + A_j)^{p_j} \quad \forall t\in [-A_j, -B_j],
\]
and where $B_j \in (1, A_j)$, $p_j\in (1, 3/2)$, $j =1,\dots, N$. The precise values of the constants $\beta_j$
and the properties of each
$\left.\psi_j\right|_{-B_j < t \leqslant \beta_j}$, $j =1,\dots, N$, are determined by the construction
that follows. Consider the ``hypersurfaces of revolution''
\[
  \hypsrf_j \defeq \{(z, w)\in \C^2 : |z|^2 + |\iprt{w}|^2 =
  C_j\psi_j(\rprt{w})^2, \ -A_j \leqslant \rprt(w) \leqslant \beta_j\},
\]
where each $C_j$ is a positive constant whose value we shall fix appropriately. Now, consider the Euclidean
unit sphere $S^3\subset \C^2$ and pick $N$ distinct points $p_1,\dots, p_N\in S^3$, $N\geqslant 2$. Fix
unitary transformations $\uni_j$ (relative to the standard Hermitian inner product on $\C^2$) such that
\[
 \uni_j(0, -1) = p_j
\]
for each $j = 1, \dots, N$. Now consider the half-spaces
$\Sigma_j \defeq \{(z, w)\in \C^2 : \rprt(w) < -1 + \delta_j\}$, where each $\delta_j > 0$ is a
\emph{small} constant. Let
\[
  \mathsf{C}_j \defeq S^3\cap \uni_j(\Sigma_j),
\]
$j = 1,\dots, N$; these are small caps on the sphere. It follows from the discussion in the previous
subsection that, by adjusting the values of the constants $B_j, \beta_j, C_j$ and $\delta_j$ introduced above
appropriately, we can define each
$\psi_j$, $j = 1,\dots, N$, in such a way that the set
\[
 \mathcal{S} \defeq \left( S^3\setminus \cup_{j = 1}^N\mathsf{C}_j \right)\bigcup
  				\left( \cup_{j = 1}^N\uni_j(\hypsrf_j) \right)
\]
is a compact topological submanifold such that $\mathcal{S}\setminus\{q_1,\dots, q_N\}$\,---\,where
$q_j \defeq \uni_j(0, -A_j)$, $j = 1,\dots, N$\,---\,is
of class $\smoo^2$ and is strongly Levi-pseudoconvex at each of its points. It is then easy to show that the
bounded component of $\C^2\setminus \mathcal{S}$ is a caltrop.
\smallskip

The next subsection is, strictly speaking, unrelated to the issue of the existence of
caltrops. But, having shown that caltrops in $\C^2$ exist, it is easy to see that the
construction above can be generalized to $\C^n$ for every $n\geqslant 2$. We would like 
to extend the Levi-form calculation in subsection~\ref{SS:caltrop_1-spike} to higher dimensions.
This will be needed in the proof of Theorem~\ref{T:visibility-caltrops}. Thus, we conclude
this section with the following:

\subsection{A Levi-form calculation for caltrops}
We would like\,---\,in proving Theorem~\ref{T:visibility-caltrops}\,---\,to observe the notation
introduced in Definition~\ref{D:caltrop}. Thus, we present the following lemma that follows
a calculation analogous to the one in subsection~\ref{SS:caltrop_1-spike}.

\begin{lemma}\label{L:leviform}
	 Let $\OM\subset \C^n$, $n\geqslant 2$, be a caltrop. Let $q$ denote one of the points
	 $\{q_1,\dots, q_N\}\subset \bdy{\OM}$ (say $q_{j^*}$) as in Definition~\ref{D:caltrop}\,---\,i.e., one of the
	 tips of a spike of $\OM$. Let $\psi:[0, A] \to [0,+\infty)$, $V\ni q$ and $p\in (1, 3/2)$ denote the data
	 associated to $q$ by Definition~\ref{D:caltrop}. 
	  Let $(z_1,\dots, z_n)$ represent the system of global holomorphic
	 coordinates centered at $q$\,($=q_{j^*}$) obtained by the transformation of the product
	 coordinates on $\C^n$ by the map $\uni_{j^*}$. Let us abbreviate
	$(z_1,\dots, z_{n-1}, z_n)$ as $(z', z_n)$. Then,
	 \begin{enumerate}[leftmargin=22pt]
	 	\item The function
		\[
		   \rho(z) \defeq \iprt(z_n)^2 + \|z'\|^2 - \psi( \rprt(z_n) )^2, \quad (z_1,\dots, z_{n-1})\in \C^{n-1}, 
		   \ z_n: 0 < \rprt(z_n) < A,
		 \]
		 is a defining function of $\uni_{j^*}(\bdy{\OM})\cap \{(z_1,\dots, z_n) \mid 0 < \rprt(z_n) < A\}$.
		 
		 \item The Levi-form of $\rho$ is given by
		 \begin{multline*}
		   \levi(\rho)(z; v) = \|\,(v_1,\dots, v_{n-1})\,\|^2 + \left(\frac{1}{2}
		   - \frac{1}{2}\big(\psi^{\prime\prime}(\rprt{z_n})\psi(\rprt{z_n}) + \psi'(\rprt{z_n})^2\big)\right)|v_n|^2 \\
		    \forall z\in \uni_{j^*}(\bdy{\OM})\cap \{(z_1,\dots, z_n) \mid 0 < \rprt(z_n) < A\} \text{ and }
		   \forall v\in \C^n.
		 \end{multline*}
	\end{enumerate}	   
\end{lemma} 
\begin{proof}
A simple calculation reveals that $d\rho(z) \neq 0$ at each 
$z\in \uni_{j^*}(\bdy{\OM})\cap \{(z_1,\dots, z_n) \mid 0 < \rprt(z_n) < A\}$. This, together with the
explicit description of each spike of $\OM$, establishes that $\rho$ is a (local) defining function
for the stated piece of $\bdy{\OM}$.

\smallskip

The calculations required for determinining the Levi-form are analogous to those in
subsection~\ref{SS:caltrop_1-spike}. Specifically:
\begin{align*}
  \pdshort{z_j\zbar_j}{2}\rho &\equiv 1 \; \; \text{for $j = 1,\dots, n-1$,}\\
  \pdshort{z_j\zbar_k}{2}\rho &\equiv 0 \; \; \text{for $j\neq k$, $1\leqslant j, k\leqslant n$}, \\
  \pdshort{z_n\zbar_n}{2}\rho(z', z_n) &= \frac{1}{2}
  	- \frac{1}{2}\left(\psi^{\prime\prime}(\rprt{z_n})\psi(\rprt{z_n}) + \psi'(\rprt{z_n})^2\right),
\end{align*}
wherever $\rho$ is of class $\smoo^2$\,---\,which is the case for each
$z : (z', z_n)\in \C^{n-1}\times \{z_n \mid 0 < \rprt(z_n) < A\}$. From this, the expression for
$\levi(\rho)(z; v)$ follows.
\end{proof}
\smallskip

\section{General properties of visibility domains}\label{S:props_visibility}

In this section, we shall demonstrate three properties of visibility domains with respect to
the Kobayashi distance. The first two results will seem to be of a slightly technical nature. However,
the three results together form the crux of the argument underlying the observation\,---\,made
in Section~\ref{S:intro}\,---\,that several results that were shown for Goldilocks domains in
\cite{Bharali_Zimmer} actually hold true for visibility domains with respect to the Kobayashi distance.

\smallskip

The proof of the first of these results, Proposition~\ref{P:iteration_seqs}, is based on an argument
developed by Karlsson in \cite{Karlsson}. 
Its near-resemblance to Theorems~\ref{T:Wolff_Denjoy_gen}~and~\ref{T:Wolff_Denjoy_special}
is suggestive. Of course, the conclusion \eqref{E:conclsn_all_z} is weaker than what constitutes a
Wolff--Denjoy-type theorem, but the domains appearing in Proposition~\ref{P:iteration_seqs}
are\,---\,in contrast to those in the above-mentioned theorems\,---\,\textbf{merely}
visibility domains with respect to the Kobayashi distance. Proposition~\ref{P:iteration_seqs}
is an indication that some of the results stated in \cite{Bharali_Zimmer} for Goldilocks domains
might be true for visibility domains with respect to the Kobayashi distance. 

\begin{proposition}\label{P:iteration_seqs}
	Let $\OM \subset \C^n$ be a visibility domain with respect to the Kobayashi distance.
	Let $F:\OM\to \OM$ be a holomorphic map. Let
	$(\nu_j)_{j \geqslant 1}$ and $(\mu_j)_{j \geqslant 1}$ be two sequences of
	positive integers with $\nu_j, \mu_j\to \infty$. Suppose   
	\begin{equation}\label{E:blowing_up}
	  \lim_{j \to \infty} \koba_{\OM}( F^{\nu_j}(o), o) = \infty \quad{and}
	  \quad \lim_{j \to \infty} \koba_{\OM}( F^{\mu_j}(o), o) = \infty
	\end{equation}
	for some $o \in \OM$. Then there exists a $\xi \in \bdy{\OM}$ such that 
	\begin{equation}\label{E:conclsn_all_z}
 	  \lim_{j \rightarrow \infty} F^{\nu_j}(z) = \xi = \lim_{j \rightarrow \infty} F^{\mu_j}(z)
 	\end{equation}
	for every $z \in \OM$. 
\end{proposition}
\begin{proof}
By \eqref{E:blowing_up} and the fact that $\OM$ is bounded, we can find a 
subsequence $(\nu_{j_{\ell}})_{\ell \geqslant 1}$ such that:
\begin{enumerate}[leftmargin=22pt]
	\item[$(a)$] $\koba_{\OM}(F^{\nu_{j_{\ell}}}(o), o) \geqslant \koba_{\OM}(F^{k}(o), o)$
	for every $k\leqslant \nu_{j_{\ell}}$, $\ell = 1, 2, 3,\dots$;
	\item[$(b)$] $F^{\nu_{j_{\ell}}}(o)\to \xi$ for some point $\xi\in \bdy{\OM}$ as $\ell\to \infty$.
\end{enumerate}
We now establish the following
\vspace{0.5mm}

\noindent{\textbf{Claim:} Let $z\in \OM$ and let $(m_j)_{j \geqslant 1}$ be a sequence of positive
integers with $m_j\to \infty$ such that $\koba_{\OM}( F^{m_j}(z), z)\to \infty$ as $j\to \infty$.
Suppose $(m_{j_{\ell}})_{\ell \geqslant 1}$ is a subsequence such that
\[
  F^{m_{j_{\ell}}}(z)\to \eta \quad\text{as $\ell\to \infty$},
\]
where $\eta$ is some point in $\bdy{\OM}$. Then $\eta = \xi$.}
\vspace{0.5mm}

\noindent{\emph{Proof of claim:} We shall assume $\xi \neq \eta$ and aim for a contradiction. 
For simplicity of notation, let us, for just this paragraph, relabel
$(\nu_{j_{\ell}})_{\ell \geqslant 1}$ as $(\nu_j)_{j \geqslant 1}$\,---\,but with the understanding
that it represents the subsequence introduced at the beginning of the proof. 
Also, relabel $(m_{j_{\ell}})_{\ell \geqslant 1}$ as $(m_j)_{j \geqslant 1}$. Pick a
sequence $i_j \to \infty$ such that $\nu_{i_j} > m_j$, $j = 1, 2, 3,\dots$
Now let $\sigma_j :[0,T_j]\to \OM$ be a
$(1,1)$-almost-geodesic with $\sigma_j(0) = F^{\nu_{i_j}}(o)$ and $\sigma_j(T_j) = F^{m_j}(z)$,
whose existence is guaranteed by Proposition~4.4 of \cite{Bharali_Zimmer}.
Since $\OM$ is a visibility domain with respect to the Kobayashi distance, and as $\xi\neq \eta$ (by
assumption), there exists an $R > 0$ so that 
\[
  \sup\nolimits_{j \geqslant 1} \koba_{\OM}(o, \sigma_j) \leqslant R,
\]
where we write $\koba_{\OM}(o, \sigma_j) \defeq \inf\{\koba_{\OM}(o, \sigma_j(t)) \mid t\in [0, T_j]\}$.
We pick some $t_j \in [0,T_j]$ such that
$\koba_{\OM}(o, \sigma_j(t_j)) \leqslant R$ , $j = 1, 2, 3,\dots$
Then, by Lemma~\ref{L:basic_but_important} we have 
\begin{align*}
  \koba_{\OM}(F^{\nu_{i_j}}(o), F^{m_j}(z))
  & \geqslant \koba_{\OM}(F^{\nu_{i_j}}(o), \sigma_j(t_j)) + \koba_{\OM}(\sigma_j(t_j), F^{m_j}(z)) - 3 \\
  & \geqslant \koba_{\OM}(F^{\nu_{i_j}}(o), o) + \koba_{\OM}(o, F^{m_j}(z))- 3- 2R.
\end{align*}
On the other hand
\[
  \koba_{\OM}(F^{\nu_{i_j}}(o), F^{m_j}(z))
  \leqslant \koba_{\OM}(F^{\nu_{i_j}-m_j}(o), o) + \koba_{\OM}(o, z)
  \leqslant \koba_{\OM}(F^{\nu_{i_j}}(o), o) + \koba_{\OM}(o, z).
\]
The first inequality is due to the triangle inequality and the fact that $F$ is contractive with respect to the Kobayashi
distance and the second is due to the property $(a)$. We conclude that
\begin{align}
   \koba_{\OM}(z, F^{m_j}(z)) &\leqslant \koba_{\OM}(F^{m_j}(z), F^{m_j}(o))
   	+ \koba_{\OM}(F^{m_j}(o), o) + \koba_{\OM}(o,z) \label{E:3-part_triangle} \\
   &\leqslant 2R + 3 + 4\koba_{\OM}(o, z), \notag
\end{align}
which produces a contradiction. 
(The reader will notice that a more efficient bound is possible above, but we opt for the 3-term upper bound in
\eqref{E:3-part_triangle} because we will
need, and refer to, the idea behind this bound later.) Hence the claim.\hfill $\blacktriangleleft$}
\smallskip

Taking $z=o$ in the above claim, and letting $(m_j)_{j \geqslant 1}$
represent any subsequence $(\nu_{j_k})_{k \geqslant 1}$ of $(\nu_j)_{j \geqslant 1}$
(respectively, $(\mu_{j_k})_{k \geqslant 1}$ of $(\mu_j)_{j \geqslant 1}$) for which
$\big( F^{\nu_{j_k}}(o) \big)_{k \geqslant 1}$ is convergent 
(respectively, $\big( F^{\mu_{j_k}}(o) \big)_{k \geqslant 1}$ is convergent), we conclude that
\[
  \lim_{j \rightarrow \infty} F^{\nu_j}(o) = \xi = \lim_{j \rightarrow \infty} F^{\mu_j}(o).
\]
Now consider $z\neq o$. Arguing as in \eqref{E:3-part_triangle} and by the fact that $F$ is contractive,
we have
\[
  \koba_{\OM}( F^{\nu_j}(z), z) \geqslant \koba_{\OM}( F^{\nu_j}(o), o) - 2\koba_{\OM}(o,z)
  \; \; \text{and} \; \;
   \koba_{\OM}( F^{\mu_j}(z), z) \geqslant \koba_{\OM}( F^{\mu_j}(o), o) - 2\koba_{\OM}(o,z).
\]
Therefore, if $(m_j)_{j \geqslant 1}$ represents any subsequence $(\nu_{j_k})_{k \geqslant 1}$ of
$(\nu_j)_{j \geqslant 1}$ (respectively, $(\mu_{j_k})_{k \geqslant 1}$ of $(\mu_j)_{j \geqslant 1}$) for which
$\big( F^{\nu_{j_k}}(z) \big)_{k \geqslant 1}$ is convergent 
(respectively, $\big( F^{\mu_{j_k}}(z) \big)_{k \geqslant 1}$ is convergent), then from the last two
inequalities and  from \eqref{E:blowing_up}, we have
\[
  \koba_{\OM}( F^{m_j}(z), z)\to \infty \quad\text{as $j\to \infty$.}
\]
We can therefore appeal again to our claim, whence, arguing as above, we have
\eqref{E:conclsn_all_z}.
\end{proof}

It turns out that in many applications of visibility, such as the Wolff--Denjoy-type theorems in this 
paper (as well as other applications which we shall address in forthcoming work), knowing that 
$\lim_{r\to 0^+}M_{\OM}(r) = 0$ is of crucial importance. This is guaranteed, by definition,
whenever $\OM$ is a Goldilocks domain. It is not clear whether this is true for visibility domains with
respect to the Kobayashi distance in general. However, for many sub-families of visibility domains, it
can be shown that $\lim_{r\to 0^+}M_{\OM}(r) = 0$. The following theorem is a result of this type.
  
\begin{theorem}\label{T:M_Omega_zero}
	Let $\OM$ be a visibility domain with respect to the Kobayashi distance that is taut.
	Then $\lim_{r\to 0^+}M_{\OM}(r) = 0$.
\end{theorem}
\begin{proof}
Assume that $M_{\OM}(r) \not\to 0$ as $r\to 0$. Since, by definition, $M_{\OM}(r)$ is monotone, this implies
that there exists a constant $\epsilon_0 > 0$ such that $M_{\OM}(r) \searrow \epsilon_0$ as
$r \searrow 0$. Thus, there exist a sequence of positive numbers $r_1 > r_2 > r_3 > \dots$ such that
$r_{\nu} \to 0$ and, for each $\nu\in \posint$, a point
$z_{\nu}\in \OM$ such that $0 < \delta_{\OM}(z_{\nu}) \leqslant r_{\nu}$ and such that:
\begin{itemize}
	\item $\epsilon_0 \leqslant M_{\OM}(r_{\nu}) < \epsilon_0 + 1/{\nu}$; and
	\item $\exists v_{\nu}\in T^{1,0}_{z_{\nu}}\OM$ satisfying $\|v_{\nu}\| = 1$ and
	\begin{equation}\label{E:M_Omega_conseq}
	  \frac{1}{\dkoba_{\OM}(z_{\nu}; v_{\nu})} >\epsilon_0 - \frac{1}{\nu}.
	\end{equation}
\end{itemize}
Owing to the definition of $\dkoba_{\OM}$, \eqref{E:M_Omega_conseq} implies that there exists, for each
$\nu\in \posint$, a holomorphic map $\varphi_{\nu}\in \hol(\unitdisk; \OM)$ satisfying
\[
  \varphi_{\nu}(0) = z_{\nu}, \quad \varphi'_{\nu}(0) \in \{t\bcdot v_{\nu} \mid t > 0\}
  \quad\text{and} \quad \|\varphi'_{\nu}(0)\| \geqslant \epsilon_0 - 1/{\nu}.
\]
Passing to a subsequence and relabelling if necessary, we may assume:
\begin{enumerate}[leftmargin=22pt]
	\item[$(a)$] there exists a point $\xi\in \bdy{\OM}$ such that $z_{\nu}\to \xi$; and
	\item[$(b)$] there exists a map $\varphi\in \hol(\unitdisk; \overline{\OM})$ such that 
	$\varphi_{\nu}\to \varphi$ uniformly on compact subsets.
\end{enumerate}
The conclusion $(b)$ is a consequence of Montel's theorem. However, as $z_{\nu}\to \xi\in \bdy{\OM}$, it
follows from the tautness of $\OM$ that $\varphi(\unitdisk)\subset \bdy{\OM}$.

\smallskip

It follows from the above discussion that $\|\varphi'(0)\| \geqslant \epsilon_0$. However, as 
$\varphi'_{\nu}\to \varphi'$ uniformly on compact sets also, and as\,---\,owing to the fact that $\OM$
is taut\,---\,$\dkoba_{\OM} : \OM\times \C^n\to [0, +\infty)$ is continuous, see
Result~\ref{R:koba_facts_miscGeom}-(\ref{Cnclsn:taut_domains_cont}), there exist a small constant $\delta_1 > 0$ and
a number $N_1\in \posint$ such that
\begin{align}
  \|\varphi'(\zeta)\|, \ \|\varphi'_{\nu}(\zeta)\| &\geqslant \epsilon_0/2
  \quad \forall \zeta\in \overline{D(0, \delta_1)} \text{ and } \forall \nu\geqslant N_1, \label{E:deriv_bound}\\
  \dkoba_{\OM}\big( \varphi_{\nu}(\zeta); \varphi'_{\nu}(\zeta) \big) &\leqslant 2/\epsilon_0
  \quad \forall \zeta\in \overline{D(0, \delta_1)} \text{ and } \forall \nu\geqslant N_1. \label{E:koba_control}
\end{align}
Let $\pi_j$ denote the projection onto the $j$-th coordinate.
By Cauchy's estimates, the magnitude of each of the derivatives of
$\pi_j\circ\varphi_{\nu}$, $j = 1,\dots, n$, $\nu\geqslant N_1$, is bounded above by a quantity
that depends only on $\sup_{x : |x| = \delta_1}|\pi_j\circ\varphi_{\nu}(x)|$, $\delta_1$,
and the order of the derivative in question, and which is
independent of $\zeta$ if $\zeta\in \overline{D(0; \delta_1/2)}$. Thus, by a standard power-series argument
and by \eqref{E:deriv_bound}, we can find a small constant $\delta_2\in (0, \delta_1/2)$ and an integer
$N_2\geqslant N_1$ so that
\begin{equation}\label{E:difference}
  \|\varphi_{\nu}(\zeta_1) - \varphi_{\nu}(\zeta_2)\| \geqslant \frac{\epsilon_0}{4}|\zeta_1 - \zeta_2|
  \quad \forall \zeta_1, \zeta_2\in \overline{D(0, \delta_2)} \text{ and } \forall \nu\geqslant N_2.
\end{equation}

\smallskip

Let us now write
\[
  \bdy{\OM}\ni \eta \defeq \varphi(\delta_2/2) \quad\text{and}
  \quad w_{\nu} \defeq \varphi_{\nu}(\delta_2/2).
\]
It follows from \eqref{E:difference} that $\xi \neq \eta$.
Clearly, $w_{\nu}\to \eta$. Let $\gamma : ([0, T], 0, T)\to (\unitdisk, 0, \delta_2/2)$ denote the
geodesic with respect to the Poincar{\'e} distance on $\unitdisk$ from $0$ to $\delta_2/2$ that
lies in $[0, 1)\subset \unitdisk$. Let us define $\sigma_{\nu} : [0, T]\to \OM$ as
$\sigma_{\nu}(t) \defeq \varphi_{\nu}\circ\gamma(t)$. We claim that each $\sigma_{\nu}$,
$\nu\geqslant N_2$, is a $(\lambda, 0)$-almost geodesic for an appropriate $\lambda\geqslant 1$.
We first note that as $\gamma$ is the restriction of a diffeomorphic
embedding of $\R$ into $\unitdisk$, there exists a constant $r_0 > 0$ such that
\begin{equation}\label{E:diffeo}
  |\gamma(s) - \gamma(t)| \geqslant r_0|s - t|
  \quad \forall s, t\in [0, T].
\end{equation}  
We now estimate, for any $s, t\in [0, T]$ and any $\nu\geqslant N_2$:
\begin{align*}
  \koba_{\OM}\big( \sigma_{\nu}(s), \sigma_{\nu}(t) \big)\,&\geqslant\,c\|\sigma_{\nu}(s) - \sigma_{\nu}(t)\| \\
  &\geqslant\,\frac{c\,\epsilon_0}{4}|\gamma(s) - \gamma(t)|
  && (\text{by \eqref{E:difference} above}) \\
  &\geqslant\,\frac{c\,\epsilon_0\,r_0}{4}|s - t|
  && (\text{by \eqref{E:diffeo} above})
\end{align*}
Here, the constant $c > 0$ in the first inequality is as given by Result~\ref{R:koba_facts_misc}.
On the other hand, by the fact that each $\varphi_{\nu}$ is contractive relative to the Kobayashi distance, we have
for any $s, t\in [0, T]$ (recall that the Poincar{\'e} distance on $\unitdisk$ is $\koba_{\unitdisk}$):
\begin{align*}
  \koba_{\OM}\big( \sigma_{\nu}(s), \sigma_{\nu}(t) \big)\,&
  \leqslant\,\koba_{\unitdisk}\big( \gamma(s), \gamma(t) \big) \\
  &=\,|s - t|.
\end{align*}
Furthermore, by \eqref{E:koba_control}, we have
\[
  \dkoba_{\OM}\big( \sigma_{\nu}(t); \sigma'_{\nu}(t) \big) \leqslant 
  \frac{2}{\epsilon_0}\sup\nolimits_{\tau\in [0, T]}|\gamma'(\tau)|.
\]
From these estimates, and by the fact that as each $\sigma_{\nu}$\,---\,being $\smoo^\infty$-smooth\,---\,is
absolutely continuous, we get that each $\sigma_{\nu}$,
$\nu\geqslant N_2$, is a $(\lambda, 0)$-almost geodesic from $z_{\nu}$ to $w_{\nu}$ with
\[
  \lambda = \max\left( 1,\,\frac{4}{c\,\epsilon_0\,r_0},\,\frac{2}
  {\epsilon_0}\sup\nolimits_{\tau\in [0, T]}|\gamma'(\tau)| \right).
\]
Since $\varphi(\unitdisk)\subset \bdy{\OM}$, it follows that given any compact subset $K\subset \OM$
there exists an integer $N_K \gg 1$ such that $\varphi_{\nu}(\,\overline{D(0, \delta_2/2)}\,)\cap K = \varnothing$
for every $\nu\geqslant N_K$. But this, together with our conclusions about
$\sigma_{\nu}$ (for  $\nu\geqslant N_2$), contradicts the fact that $\OM$ is a visibility domain
with respect to the Kobayashi distance. Hence our assumption about $M_{\OM}$ must be false.
\end{proof}

The last result in this section is one whose conclusion identifies a property that is possessed by domains with
smooth boundaries that are ``sufficiently curved'' in a certain sense. However, Theorem~\ref{T:bdy-vld_lmts_cnstnt}
establishes that any taut visibility domain\,---\,whose boundary is, in general far less well-behaved\,---\,also has the
desirable property alluded to. 

\begin{theorem} \label{T:bdy-vld_lmts_cnstnt}
	Let $X$ be a connected complex manifold and
	let $\OM \subset \C^n$ be a visibility domain with respect to the Kobayashi distance that is
	taut. Suppose $(\varphi_{\nu})_{\nu \geqslant 1}$ is a sequence in $\hol(X; \OM)$
	that converges uniformly on compacts of $X$ to a holomorphic map $\psi: X\to \bdy \OM$. Then $\psi$ is a constant map.
\end{theorem}
\begin{proof}
Fix $x \in X$. For any $f\in \hol(X; \OM)$, let $f'$ denote the holomorphic total derivative of $f$.
Since $(\varphi_{\nu})_{\nu \geqslant 1}$ converges uniformly on compacts to $\psi$, it follows
that $\varphi'_{\nu}(x) \to \psi'(x)$ (the easiest way to understand this is to equip $X$ with some hermitian metric; the choice
of metric is irrelevant to the proof). Fix a vector $v_0 \in (T^{1,0}_{x} X)\setminus \{0\}$. 
We claim that, given a $\nu\in \posint$, $\|\varphi'_{\nu}(x)v_0\|
\leqslant \dkoba_X(x, v_0)M_{\OM}\big( \delta_{\OM}( \varphi_{\nu}(x) ) \big)$. 
There is nothing to prove if $v_0\in {\rm Ker}(\varphi'_{\nu}(x))$. Thus, assume that 
$v_0 \notin {\rm Ker}(\varphi'_{\nu}(x))$. We estimate:
\[
  \frac{ \|\varphi'_{\nu}(x)v_0\| }{ \dkoba_{\OM}\big( \varphi_{\nu}(x); \varphi'_{\nu}(x)v_0 \big) } =
\frac{1}{ \dkoba_{\OM}\left( \varphi_{\nu}(x); \tfrac{\varphi'_{\nu}(x)v_0}{\|\varphi'_{\nu}(x)v_0\|} \right) } \leqslant
M_{\OM}\big( \delta_{\OM}( \varphi_{\nu}(x) ) \big).
\]
The inequality on the right side is due to the definition of $M_{\OM}$. Therefore
\[
  \|\varphi'_{\nu}(x)v_0\| \leqslant  \dkoba_{\OM}\big( \varphi_{\nu}(x); \varphi'_{\nu}(x)v_0 \big)
  M_{\OM}\big( \delta_{\OM}( \varphi_{\nu}(x) ) \big) 
  \leqslant \dkoba_X(x; v_0) M_{\OM}\big( \delta_{\OM}( \varphi_{\nu}(x) ) \big),
\]
which is the desired claim.
The second inequality is due the metric-decreasing property of holomorphic maps. 
By hypothesis, $\delta_{\OM}(\varphi_{\nu}(x)) \to 0$ as $\nu\to \infty$. Since $\OM$ is taut, it follows from
Theorem~\ref{T:M_Omega_zero} that $M_{\OM}\big( \delta_{\OM}( \varphi_{\nu}(x) )\big) \to 0$
as $\nu \to \infty$. Therefore, from the last inequality, we see that $\varphi'_{\nu}(x)v_0 \to 0$ as $\nu \to \infty$.
This in turn implies that $\psi'(x)v_0=0$. Now $v _0\in (T^{1,0}_{x} X)\setminus \{0\}$ was arbitrary, whence we get
$\psi'(x) \equiv 0$. As the above $x\in X$ was arbitrary, and as $X$ is connected, it follows that $\psi$ is a constant. 
\end{proof}
\smallskip

\section{The proof of Theorem~\ref{T:gen_visibility-lemma}}\label{S:gen_visibility-lemma}
This section is devoted to proving Theorem~\ref{T:gen_visibility-lemma}. To do so, we first need a technical lemma.

\begin{lemma} \label{L:cntrl_MOmega_lmm}
	Let $f$ be as in Theorem~\ref{T:gen_visibility-lemma}. Fix constants $\lambda\geqslant 1$ and
	$\kappa\geqslant 0$. Then, given $\epsilon > 0$, there exist constants $-\infty < a' < b' < +\infty$
	such that
	\begin{align*}
	  \int_{-\infty}^{a'} M_{\OM}\!\Big( \frac{1}{f^{-1}\big( (1/2 \lambda)|t|-(\kappa/2) \big)} \Big) dt &< \epsilon, \\
	  \int_{b'}^{+\infty} M_{\OM}\!\Big( \frac{1}{f^{-1}\big( (1/2 \lambda)t-(\kappa/2) \big)} \Big) dt &< \epsilon.
	\end{align*}
\end{lemma}
\begin{proof}
The result is a consequence of the change-of-variable formula, using
\[
  r \defeq \frac{1}{f^{-1}\big( (1/2 \lambda)|t| - (\kappa/2) \big)} 
\]
for the first integral, and
\[
  r \defeq \frac{1}{f^{-1}\big( (1/2 \lambda)t - (\kappa/2) \big)} 
\]
for the second. We omit the routine computations that these changes of variable necessitate. The inequalities
follow from the integrability condition \eqref{E:intgr_cndn_M}.
\end{proof}

We are now in a position to give the:

\begin{proof}[Proof of Theorem~\ref{T:gen_visibility-lemma}]
We proceed by contradiction. Assume thus that there exist constants $\lambda \geqslant 1$ and $\kappa \geqslant 0$,\
a pair of distinct points $\xi,\eta \in \bdy \OM$, neighbourhoods $V$ and $W$ of $\xi$ and $\eta$, respectively, in
$\overline{\OM}$ with $\overline{V} \cap \overline{W} = \varnothing$, 
and a sequence $( \sigma_{\nu} )_{\nu \geqslant 1}$ of $(\lambda, \kappa)$-almost-geodesics,
$\sigma_{\nu} : [a_{\nu},b_{\nu}] \to \OM$, such that $\sigma_{\nu}(a_{\nu}) \in V$ and
$\sigma_{\nu}(b_{\nu}) \in W$ for
all $\nu$ and such that
\[
  \max_{t \in [a_{\nu},b_{\nu}]}\dtb{\OM}(\sigma_{\nu}(t)) \to 0 \; \; \text{as} \; \; \nu \to \infty.
\]
By re-parametrizing, we can assume that, for all $\nu$, $a_{\nu} \leqslant 0 \leqslant b_{\nu}$ and that
\[
  \dtb{\OM}(\sigma_{\nu}(0)) = \max_{t \in [a_{\nu},b_{\nu}]} \dtb{\OM}(\sigma_{\nu}(t)).
\]
By Result~\ref{R:Lipschitz}, there exists a $C < \infty$ such that every $\OM$-valued
$(\lambda,\kappa)$-almost-geodesic is $C$-Lipschitz with respect to the Euclidean distance. Therefore, by using the
Arzela--Ascoli theorem and passing to an appropriate subsequence, we may assume:
\begin{itemize}
	\item $a_{\nu} \to a \in [-\infty,0]$ and $b_{\nu} \to b \in [0,+\infty]$; \smallskip
	\item $(\sigma_{\nu})_{\nu \geqslant 1}$ converges locally uniformly on $(a,b)$ to a continuous map $\sigma: (a,b) \to 
	\overline{\OM}$; \smallskip
	\item $( \sigma_{\nu}(a_{\nu}) )_{\nu \geqslant 1}$ converges to $\xi' \in \overline{V}$; and \smallskip
	\item $( \sigma_{\nu}(b_{\nu}) )_{\nu \geqslant 1}$ converges to $\eta' \in \overline{W}$.
\end{itemize}
Clearly, $\xi' \neq \eta'$ because $\overline{V} \cap \overline{W} = \varnothing$. We can conclude from the fact that
\[
\|\sigma_{\nu}(a_{\nu})-\sigma_{\nu}(b_{\nu})\| \leqslant C(b_{\nu}-a_{\nu}) \quad \forall \nu \in \posint
\]
that $a<b$.
\smallskip

\noindent{\textbf{Claim:} If $\theta: [s_1,s_2] \to \OM$ is a $(\lambda,\kappa)$-almost-geodesic,
then for almost every $t \in [s_1,s_2]$, $\|\theta'(t)\| \leqslant \lambda
M_{\OM}( \dtb{\OM}(\theta(t)) )$.}
\vspace{0.5mm}

\noindent{\emph{Proof of claim:} By the definition of a $(\lambda,\kappa)$-almost-geodesic we have
$\kappa_{\OM}(\theta(t),\theta'(t)) \leqslant \lambda$ for almost every $t \in [s_1,s_2]$.
If $\theta'(t)=0$, then the claim is 
trivially true. If $\theta'(t) \neq 0$, we have
\[
\kappa_{\OM}\left( \theta(t),\frac{\theta'(t)}{\|\theta'(t)\|} \right) \leqslant \frac{\lambda}{\|\theta'(t)\|}.
\]
So
\[
  \|\theta'(t)\| \leqslant \lambda \cdot \frac{1}{\kappa_{\OM}\!\left( \theta(t),\tfrac{\theta'(t)}{\|\theta'(t)\|} \right)}
			 \leqslant \lambda M_{\OM}\big( \dtb{\OM}(\theta(t)) \big).  
  \tag*{$\blacktriangleleft$}
\]}
\smallskip

We first assert that $\sigma : (a,b) \to \overline{\OM}$ is constant. To prove this, we use the fact that $M_{\OM}(t)
\searrow 0$ as $t \searrow 0$. With this, the proof proceeds exactly along the lines of the proof of Claim~1 in
\cite[Section~5]{Bharali_Zimmer}. Hence, we omit the proof.
\smallskip

We shall now show that $\sigma$ is \emph{not} constant. Our argument involves the study of
two cases.
\smallskip

\noindent{\emph{Case~1.} Both $a$ and $b$ are finite.}
\vspace{0.5mm}

\noindent{In this case, we first define the $C$-Lipschitz maps
$\widetilde{\sigma}_{\nu} : [a, b] \to \OM$ obtained by restricting each $\sigma_{\nu}$
to $[a_{\nu}, b_{\nu}]\cap [a, b]$ and then extending the restricted map
continuously to $[a, b]$ by defining the extension to be a constant on the intervals $[a, a_\nu]$ and
$[b_\nu, b]$ whenever $a < a_{\nu}$ or $b_{\nu} < b$. We can then infer by a standard argument that
$\sigma$ extends to a continous map $\widetilde{\sigma}: [a, b] \to \overline{\OM}$. We have
$\widetilde{\sigma}(a) = \xi' \neq \eta' = \widetilde{\sigma}(b)$. By continuity of $\widetilde{\sigma}$,
it follows that $\left.\widetilde{\sigma}\right|_{(a, b)}$ is non-constant.}
\smallskip

\noindent{\emph{Case~2.} Either $a = -\infty$ or $b = +\infty$.}
\vspace{0.5mm}

\noindent{We make a couple of preliminary observations. For every $\nu \in \posint$  and every
$t \in [a_{\nu},b_{\nu}]$,
\begin{align*}
  \frac{1}{\lambda}|t|-\kappa \leqslant \koba_{\OM}(\sigma_{\nu}(0),\sigma_{\nu}(t))
  &\leqslant \koba_{\OM}(\sigma_{\nu}(0),z_0) + \koba_{\OM}(z_0,\sigma_{\nu}(t)) \\
  &\leqslant 2 f\left( \frac{1}{\dtb{\OM}(\sigma_{\nu}(t))} \right),
\end{align*}
because $\dtb{\OM}(\sigma_{\nu}(0)) \geqslant \dtb{\OM}(\sigma_{\nu}(t))$.}
\smallskip

Let us first consider the case when $b = +\infty$. By the properties of the
sequence $( \sigma_{\nu} )_{\nu \geqslant 1}$, it follows that there exists $N \in \posint$
and a constant $B \gg 1$ such that
\[
  \frac{1}{2 \lambda}|t|-\frac{\kappa}{2} \in \range(f) \quad \forall t \in (B, b_{\nu} ] \text{ and }
  \forall \nu \geqslant N.
\]
Thus, by the fact that $f$ is strictly increasing, we get:
\begin{equation}\label{E:ineq_involving_f-inv_1}
  f^{-1}\left( \frac{1}{2 \lambda}|t|-\frac{\kappa}{2} \right) \leqslant \frac{1}{\dtb{\OM}(\sigma_{\nu}(t))}
  \quad \forall t \in (B, b_{\nu} ]  \text{ and }  \forall \nu \geqslant N,
\end{equation}
in case $b = +\infty$. If $a = -\infty$, we can argue in exactly the same way to find a constant $A \gg 1$
such that
\begin{equation}\label{E:ineq_involving_f-inv_2}
  f^{-1}\left( \frac{1}{2 \lambda}|t|-\frac{\kappa}{2} \right) \leqslant \frac{1}{\dtb{\OM}(\sigma_{\nu}(t))}
  \quad \forall t \in [a_{\nu}, -A)  \text{ and } \forall \nu \geqslant N
\end{equation}
(where $N$ is exactly as above).
\smallskip

At this juncture, we shall assume that $a = -\infty$ and $b = +\infty$. This is the principal sub-case; we shall
merely indicate the changes that would be needed in the argument that follows in case either \textbf{one} of
$a$ or $b$ is finite. With this assumption, we have, by monotonicity of $M_{\OM}$ and 
from \eqref{E:ineq_involving_f-inv_1} and \eqref{E:ineq_involving_f-inv_2}:
\[
  M_{\OM}\big( \dtb{\OM}(\sigma_{\nu}(t)) \big)
  \leqslant M_{\OM}\Big( \frac{1}{f^{-1}\big( (1/2 \lambda)|t|- (\kappa/2) \big)} \Big)
\]
for evey $t \in [a_{\nu}, -A)\cup (B, b_{\nu} ]$ and for every $\nu \geqslant N$.
So, finally, by our claim above, we conclude that
\begin{align} 
  \|\sigma'_{\nu}(t)\| &\leqslant \lambda M_{\OM}\big( \dtb{\OM}(\sigma_{\nu}(t)) \big) \notag \\
  				    &\leqslant \lambda M_{\OM}\Big( \frac{1}{f^{-1}\big( (1/2 \lambda)|t|- (\kappa/2) \big)} \Big)
  \quad\text{for a.e. $t \in [a_{\nu}, -A)\cup (B, b_{\nu} ]$ and
  $\forall\nu \geqslant N$}. \label{E:upper_bound_sigma_nu_prime}
\end{align}
\smallskip

Using Lemma~\ref{L:cntrl_MOmega_lmm}, we choose $a'\in (-\infty, -A)$ and $b'\in (B, +\infty)$ such that
\[
  \lambda\!\int_{-\infty}^{a'}\!M_{\OM}\Big( \frac{1}{f^{-1}\big( (1/2 \lambda)|t|-(\kappa/2) \big)} \Big) dt +
  \lambda\!\int_{b'}^{+\infty}\!M_{\OM}\Big( \frac{1}{f^{-1}\big( (1/2 \lambda)t-(\kappa/2) \big)} \Big) dt <
\|\xi'-\eta'\|.
\]
Then
\begin{align*}
  \|\sigma(b')-\sigma(a')\|
  &= \lim_{\nu \to \infty} \|\sigma_{\nu}(b')-\sigma_{\nu}(a')\| \\ 
  &\geqslant \limsup_{\nu \to \infty} \big( \|\sigma_{\nu}(b_{\nu})-\sigma_{\nu}(a_{\nu})\| - 
  				  \|\sigma_{\nu}(a_{\nu})-\sigma_{\nu}(a')\| - \|\sigma_{\nu}(b_{\nu})-\sigma_{\nu}(b')\| \big) \\
  &\geqslant  \|\xi'-\eta'\| - \limsup_{\nu \to \infty} \int_{a_{\nu}}^{a'} \|\sigma'_{\nu}(t)\| dt
                  - \limsup_{\nu \to \infty} \int_{b'}^{b_{\nu}} \|\sigma'_{\nu}(t)\| dt \\
  &\geqslant \|\xi'-\eta'\| - \limsup_{\nu \to \infty} \ \lambda\!\int_{a_{\nu}}^{a'} M_{\OM}
  		    \Big(\frac{1}{f^{-1}\big( (1/2 \lambda)|t|-(\kappa/2) \big)} \Big) dt \\
  &- \limsup_{\nu \to \infty} \ \lambda\!\int_{b'}^{b_{\nu}} M_{\OM}
  					   \Big( \frac{1}{f^{-1}\big( (1/2 \lambda)t - (\kappa/2) \big)} \Big) dt
  \tag*{\text{(using \eqref{E:upper_bound_sigma_nu_prime})}} \\
  &=\|\xi'-\eta'\| - \lambda\!\int_{-\infty}^{a'}\!M_{\OM}\Big( \frac{1}{f^{-1}
       \big( (1/2\lambda)|t|-(\kappa/2) \big)} \Big) dt \\
  &- \lambda\!\int_{b'}^{+\infty}\!M_{\OM} 
                                 \Big( \frac{1}{f^{-1}\big( (1/2 \lambda)t - (\kappa/2) \big)} \Big) dt\,>\,0.                                                    
\end{align*}
This shows that $\sigma$ is not constant.
\smallskip

If $a$ is finite, then by an analogue of the argument described in Case~1 (by constructing
auxiliary maps that are $C$-Lipschitz on $[a, 0]$), we infer that $\sigma$ extends to
a continuous map $\widetilde{\sigma} : [a, +\infty)$. We now estimate
$\|\sigma(b') - \widetilde{\sigma}(a)\|$\,---\,with $b'\in (B, +\infty)$ chosen appropriately so that 
we can argue as in the previous paragraph\,---\,to get $\|\sigma(b') - \widetilde{\sigma}(a)\| > 0$.
An analogous description can be given for the argument in case $b$ is finite. This completes the argument
for Case~2, with the conclusion that $\sigma$ is not constant.
\smallskip

This last assertion produces a contradiction. Thus, the assumption made at the beginning must be
false, which completes the proof.
\end{proof}

\section{A family of planar comparison domains}\label{S:comparison_dom}
In this section, we take the first step in showing that caltrops have the properties
stated in Theorem~\ref{T:gen_visibility-lemma}. The essential idea is as follows: we first explicitly calculate the
Kobayashi distance on a model planar domain $D$. Then, given a caltrop $\OM\subset \C^n$, $n\geq 2$, we shall
affinely embed copies of $D$ into $\OM$ in such a way that every point of $\OM$ that
is sufficiently close to $\bdy\OM$ is contained in one of these embedded domains. Then, the
distance-decreasing property of holomorphic mappings for the Kobayashi distance could be used
to estimate the Kobayashi distance on $\OM$.

\smallskip

Given the geometry of the boundary of a caltrop, the model comparison domain $D$ that we need
will be bounded, symmetric about the real axis, have $0$ as a boundary point and the tip of an 
outward-pointing cusp. In fact, it will be useful to construct a family of
model planar domains having the latter properties. To this end, given $a, h>0$, define the following domains 
in $\C$:
\begin{equation}
S_{a,h} \defeq \{ z=x+iy \in \C \mid x > a, \, -h < y < h \}.
\end{equation}
Let us denote the domains that we are interested in by $\cspdmaltalt$, where $\cspdmaltalt$ is the image of $S_{a,h}$
under the following biholomorphisms, composed in the order given below:
\begin{align*}
\inv(z) &\defeq 1/z \quad \forall z \in \C \setminus \{0\}, \\
\phi_{\alpha}(z) &\defeq z^{\alpha} \quad \forall z \in  \inv(S_{a, h}).
\end{align*}
Here $\alpha$ is a real number greater than 1, and $a, h > 0$ are such that $\phi_{\alpha}$ 
is in fact a biholomorphism. That $a, h > 0$ can be so chosen follows from an elementary calculation.
Specifically, we compute:
\[
\inv(S_{a,h}) = \big( \C \setminus \overline{D(-i/2h,1/2h)}\,\big) \cap 
			\big( \C \setminus \overline{D(i/2h,1/2h)}\,\big) \cap D(1/2a,1/2a).
\] 
Let us denote $\inv(S_{a,h})$ by $T_{a,h}$.

\smallskip

We make a simple observation which will be useful in the proposition below. The region $T_{a,h}$ contains $0$ in its
boundary and has a quadratic cusp at $0$. This means that there exist constants $c_1, c_2 >0$
such that, for every $z \in \bdy T_{a,h}$,
\begin{align}
c_1\rprt(z)^2 &\leqslant \iprt(z) \leqslant c_2 \rprt(z)^2, \text{ or } \notag \\
-c_2\rprt(z)^2 &\leqslant \iprt(z) \leqslant -c_1 \rprt(z)^2, \label{E:constants_qcusp}
\end{align}
depending on whether $\iprt(z)\geqslant 0$ or $\iprt(z)\leqslant 0$, provided $\rprt(z)$ is sufficiently small.
 In fact, by straightforward calculations we see that for some $\delta > 0$ sufficiently small,
$\bdy T_{a,h} \cap \{z \in \C \mid 0 \leqslant \rprt(z) \leqslant \delta\} = \grph(f) \cup \grph(-f)$, where
\begin{equation} \label{E:fn_dfng_bdry_of_T(a,h)_near_0}
f(x) = hx^2 + O(x^4) \; \text{ as } x \to 0^{+},
\end{equation}  
with the understanding that $z= x + iy$. In this section, $\grph(\cdot)$ will denote the graph of a specified function.
  
\smallskip

The following proposition describes the features of the (family of) domains $\cspdmaltalt$ that will be relevant to
estimating Kobayashi distances\,---\,in the manner hinted at above\,---\,on caltrops.

\begin{proposition}\label{P:cusp_domain_props}
	Fix $\alpha > 1$ and let $T_{a, h}$ and $\cspdmaltalt$ be as described above\,---\,with $a, h > 0$ appropriately
	chosen. Set $p \defeq (1+\alpha)/\alpha$. Then:
	\begin{enumerate}
 		\item \label{Cnclsn:cusp_domain_bdry} There exist constants $\epsilon, C_1, C_2 > 0$ such that, for every $z
 		\in \bdy\cspdmaltalt \cap \{z \in \C \mid 0 \leqslant \rprt(z) \leqslant \epsilon\}$, we have
		\begin{align*}
		  C_1\rprt(z)^p &\leqslant \iprt(z) \leqslant C_2 \rprt(z)^p, \text{ or } \\
		  -C_2 \rprt(z)^p &\leqslant \iprt(z) \leqslant -C_1 \rprt(z)^p,
		\end{align*}
		depending on whether $\iprt(z)\geqslant 0$ or $\iprt(z)\leqslant 0$.
		\item \label{Cnclsn:cusp_domain_prmtrs} If we fix a constant $M \geqslant 2$, then we can
		choose an $\epsilon > 0$ sufficiently small so that the inequalities in
		\eqref{Cnclsn:cusp_domain_bdry} hold true with $C_2 \defeq Mh\alpha$.
		Furthermore, for a given $\alpha>1$ and $h>0$, this choice of $\epsilon$ decreases
		as $a \nearrow +\infty$.
 		\item \label{Cnclsn:cusp_domain_koba_est} Fix some point $x_0\in \cspdmaltalt\cap \R$. There exists
		a constant $C > 0$, which depends on $x_0$, such that
 		\[
 		  \koba_{\cspdmaltalt}(x_0,x) \leqslant C + \frac{\pi}{4h}x^{-1/\alpha} \quad \forall x\in (0, x_0).
		\]
	\end{enumerate}
\end{proposition}

\begin{proof}
To prove \eqref{Cnclsn:cusp_domain_bdry}, we must examine the image of $T_{a, h}$ under $\phi_{\alpha}$ close to 
$0\in \bdy T_{a, h}$. Let $c_1, c_2$ be the constants given by \eqref{E:constants_qcusp},
and let the function $f$ be as introduced just prior to \eqref{E:fn_dfng_bdry_of_T(a,h)_near_0}.
We examine the images of $\grph(f)$ and $\grph(-f)$ under 
$\phi_{\alpha}$. Let us, for example, examine the image of $\grph(f)$ under $\phi_{\alpha}$. An arbitrary element of
$\grph(f)$ that is close to $0$ can be written as $x+iy$, where $x \geqslant 0$ and $c_1 x^2 \leqslant y \leqslant c_2 x^2$.
For $x > 0$ and sufficiently small, we compute:
\begin{align*}
  \phi_{\alpha}(z) &= (x+iy)^{\alpha} \\
  &= x^{\alpha} \bigg( 1 + \sum_{j=1}^{\infty} \frac{ (-1)^j }{ (2j)! } \ \prod_{\nu=0}^{2j-1}\!(\alpha-\nu)
  	\frac{y^{2j}}{x^{2j}} \bigg) + i x^{\alpha} \bigg( \sum_{j=0}^{\infty} \frac{ (-1)^j }{ (2j\!+\!1)! } \
  	\prod_{\nu=0}^{2j}\!(\alpha-\nu) \frac{ y^{2j+1} }{ x^{2j+1} } \bigg).
\end{align*}

Using the fact that $c_1x^2 \leqslant y \leqslant c_2 x^2$, it is easy to see that
\begin{align*}
  \rprt(\phi_{\alpha}(z)) &= x^\alpha + O(x^{2+\alpha}), \\
  c_1\alpha x^{1+\alpha}(1 - O(x^2)) \leqslant \iprt(\phi_{\alpha}(z)) &\leqslant c_2\alpha x^{1+\alpha}(1 + O(x^2))
\end{align*}
for $z = x + iy \in \grph(f)$ and for $x > 0$ sufficiently small. 

\smallskip 

It follows from this that we can find constants $\epsilon, C_1, C_2 > 0$ such that for all $w \in \bdy \cspdmaltalt$ with
$\rprt(w)
\leqslant \epsilon$ and $\iprt(w) \geqslant 0$,
\[
  C_1 \big( \rprt(w) \big)^p \leqslant \iprt(w) \leqslant C_2 \big( \rprt(w) \big)^p.
\]
From this and the fact that, if $z\in \bdy T_{a,h} \cap \{ \iprt(z) \leqslant 0 \}$, then
$z\in \grph(-f)$ (when $\rprt(z)$ is sufficiently small), part~\eqref{Cnclsn:cusp_domain_bdry} follows.

\smallskip

Part~\eqref{Cnclsn:cusp_domain_prmtrs} is elementary and follows from the manner in which $\domain(f)$, by construction,
depends on $a$, from \eqref{E:fn_dfng_bdry_of_T(a,h)_near_0}, and from the estimates in the last paragraph.

\smallskip  

We now address the Kobayashi-distance inequality that we need. We have a biholomorphism $\Phi_{\alpha,a,h}$ from
$\cspdmaltalt$ onto $\unitdisk$, given by
\[
  \Phi_{\alpha,a,h} = f_4 \circ f_3 \circ f_2 \circ f_1 \circ \inv
  \circ \big(\!\left.\phi_{\alpha}\right|_{T_{a,h}}\big)^{-1},
\] 
where
\begin{align*}
  f_1(z) &= i(z-a) \quad \forall \, z \in S_{a,h}, \\
  f_2(z) &= \frac{\pi z}{2h} \quad \forall \, z \in \{w \in \C \mid -h < \rprt(w) < h, \, \iprt(w) > 0\}, \\
  f_3(z) &= \sin(z) \quad \forall \, z \in \{w \in \C \mid -\pi/2 < \rprt(w) < \pi/2, \, \iprt(w) > 0\}, \\
  f_4(z) &= \frac{z-i}{z+i} \quad \forall \, z \in \{ w \in \C \mid \iprt(w) > 0 \}. 
\end{align*}
The explicit expression for $\Phi_{\alpha,a,h}$ is
\begin{equation} \label{E:xplct_xprssn_Phi_alpha_a_h}
 \Phi_{\alpha,a,h}(z) = \frac{ \sin\left( \tfrac{\pi i}{2h} \big( \tfrac{1}{z^{1/\alpha}}-a \big) \right) - i }{
 \sin\left( \tfrac{\pi i}{2h} \big( \tfrac{1}{z^{1/\alpha}}-a \big) \right) + i } \quad \forall \, z \in \cspdmaltalt.
\end{equation}
Observe that  $\Phi_{\alpha,a,h}$ maps the 
closed and bounded interval $\overline{\cspdmaltalt \cap \R}$
homeomorphically onto $[-1,1]$. 
Furthermore, it is easy to check that $\Phi_{\alpha,a,h}$ maps the point
\begin{equation}\label{E:base-pt_o}
  o_{\alpha, a, h} \defeq o \defeq 1/{\big( (2h/\pi)\log(\sqrt{2}+1)+a \big)}^{\alpha}
\end{equation}
of $\cspdmaltalt$ to $0$ and that
if $x \in \cspdmaltalt \cap \R$ is less than $o$ then $\Phi_{\alpha,a,h}(x) \in (0, 1)$.
Therefore, for all such $x$,
\[
  \koba_{\cspdmaltalt}(o, x) = \koba_{\unitdisk}\big( 0, \Phi_{\alpha,a,h}(x) \big)
  = \frac{1}{2} \log \left(\frac{1+\Phi_{\alpha,a,h}(x)}{1-\Phi_{\alpha,a,h}(x)} \right).
\]
Using \eqref{E:xplct_xprssn_Phi_alpha_a_h}, we obtain
\begin{align*}
  \frac{1}{2} \log\left( \frac{1+\Phi_{\alpha,a,h}(x)}{1-\Phi_{\alpha,a,h}(x)} \right)
  &= \frac{1}{2} \log\left( \exp\left\{\frac{\pi}{2h}\Big( \frac{1}{x^{1/\alpha}}\!-\!a \Big) \right\}
  	 - \exp\left\{ -\frac{\pi}{2h} \Big( \frac{1}{x^{1/\alpha}}\!-\!a \Big)\right\}\right) - \frac{\log 2}{2}  \\
  &\leqslant \frac{1}{2} \log\left( \exp\left\{ \frac{\pi}{2h} \Big( \frac{1}{x^{1/\alpha}}\!-\!a \Big) \right\} \right) \\
  &\leqslant \frac{\pi}{4 h}x^{-1/\alpha}.
\end{align*}
From this and the triangle inequality, \eqref{Cnclsn:cusp_domain_koba_est} of our proposition follows.
\end{proof}

The next few lemmas establish some basic observations that will\,---\,given a
caltrop $\OM\subset \C^n$, $n\geq 2$\,---\,enable us to affinely embed copies of $\cspdmaltalt$, for
suitable choices of the parameters $\alpha$, $a$ and $h$, into $\OM$ in the manner hinted at in the
beginning of this section. (The actual estimates showing that caltrops possess the properties stated in the
General Visibility Lemma will be obtained in the next section.) A note about
our notation: in the lemmas that follow, the point $o$ will be as introduced in \eqref{E:base-pt_o}, and
will be associated to the specific $\cspdmaltalt$ occurring in each lemma. Also, the lemmas below hold true
for the parameter $p\in (1, 2)$, and will be proved as such. In the next section, where caltrops make an appearance,
we shall restrict $p$ to $(1, 3/2)$.

\begin{lemma} \label{L:elem_fn_ineq}
	Suppose $\epsilon>0$ and $\phi:[0,\epsilon) \to \R$ is a continuous, strictly increasing function that is differentiable
	on $(0,\epsilon)$, such that $\phi'$ is increasing and such that $\phi(0)=0$. Then, for every
	 $(x,y) \in [0, +\infty)\times [0, +\infty)$ such that $x+y<\epsilon$, $\phi(x+y) \geqslant \phi(x)+\phi(y)$. 
\end{lemma}
The proof of the above lemma is an elementary exercise in calculus.

\begin{lemma} \label{L:ess_planar_lmm}
	Let $\psi:[0, A] \to [0,+\infty)$ be a continuous function that is $\smoo^2$ on $(0, A)$, where $A > 0$. Let
	$p \in (1, 2)$. Assume furthermore that
	\begin{itemize}
		\item there exists a constant $C > 1$ such that
		\[
		  (1/C)x^p \leqslant \psi(x) \leqslant Cx^p \quad \forall \, x \in [0, A];
		\]
		\item $\psi$ is strictly increasing; and
		\item $\psi'$ is increasing on $(0, A)$.	
	\end{itemize}
	Write $\bldblkdm \defeq \{ z \in \C \mid 0 < \rprt(z) < A, \ |\iprt(z)| < \psi( \rprt(z) ) \}$. Then there
	exist a constant $B \in (0,A)$, a compact subset 
	 $K$ that intersects $\{z \in \C \mid \rprt(z) = A\}$ and
	such that $K\setminus \{z \in \C \mid \rprt(z) = A\} \varsubsetneq \bldblkdm$,
	and constants $a,h > 0$ such that for each $x+iy \in \bldblkdm$ with $x \leqslant B$, we have
	\begin{enumerate}
		\item \label{Cnclsn:ess_planar_incl} $( \psi^{-1}(|y|)+iy ) + \cspdmalt \subseteq \bldblkdm$;
		\item \label{Cnclsn:ess_planar_x-left} $\psi^{-1}(|y|) + o > x$;
		\item \label{Cnclsn:ess_planar_o-inside} $( \psi^{-1}(|y|)+iy ) + o \in K$; and 
		\item \label{Cnclsn:ess_planar_dist} $\delta_{\bldblkdm}(x+iy) \leqslant |\psi^{-1}(|y|) - x|$.
	\end{enumerate}	
\end{lemma}
\begin{proof}
It follows from Proposition~\ref{P:cusp_domain_props} and the observation made prior to it that we may fix a constant
$M \geqslant 2$ such that for every $\alpha>1$ and every $a,h>0$ there exists an $\epsilon \equiv
\epsilon(\alpha,a,h)>0$ such that
\[
  \cspdmaltalt \subset \{ w \in \C \mid 0 < \rprt(w) < \epsilon, \, |\iprt(w)| < M h \alpha 
  (\rprt(w))^{(1+\alpha)/\alpha} \} \defines S^{\alpha,a,h}
\]
and such that, for any given $\alpha>1$ and $h>0$, $\epsilon\to 0$ as $a\to +\infty$. We let $\alpha \defeq
1/(p-1)$. We note that, by the geometry of $\cspdmaltalt$,
the $\epsilon$ with the above properties does not decrease as we decrease $h$. 
Hence, we can choose $a$ and $h$ such 
that $\epsilon < A/2$ and $Mh\alpha < 1/C$. Now, fix a constant $B$, $0<B<A$ so that 
\[
  B < \min(o, (\epsilon/2) ). 
\] 
	
\smallskip
	
Let $z = x+iy \in \bldblkdm$ and $x \leqslant B$. We consider the set $\cspdmaltalt + ( \psi^{-1}(|y|)+iy )$.
An arbitrary element of this set is of the form
$(\psi^{-1}(|y|)+s)+i(y+t)$, where $s+it \in \cspdmaltalt$. Since $\cspdmaltalt \subset S^{\alpha,a,h}$ by
construction, $0 < s
<\epsilon$ and $|t| < M h \alpha s^p$. This element is in $\bldblkdm$ if and only if
\[
  0 < \psi^{-1}(|y|)+s < A \quad\text{and}
  \quad |y+t| < \psi\big( \psi^{-1}(|y|)+s \big).
\] 
Now $0 \leqslant \psi^{-1}(|y|) < x \leqslant B$. By our choice of $B$, we have
$0 < \psi^{-1}(|y|) + s < (\epsilon/2) + \epsilon < A$.

\smallskip

Thus, to establish part~\eqref{Cnclsn:ess_planar_incl}, we must show that $|y+t| < \psi\big( \psi^{-1}(|y|) + s \big)$. As
$\bldblkdm$ is symmetric about the real axis, it suffices to deal with the case $y \geqslant 0$, $t \geqslant 0$.
Notice that $\psi$
satisfies the hypothesis of Lemma~\ref{L:elem_fn_ineq}. We have
\begin{align*}
  \psi\big( \psi^{-1}(|y|)+s \big) 
  &\geqslant y + \psi(s)  & &\text{(by Lemma~\ref{L:elem_fn_ineq})}\\
  &\geqslant y + (1/C)s^p & &\text{(by hypothesis).}
\end{align*}
Recall that $Mh\alpha < 1/C$. 
Therefore,
\[
  |y+t|\,=\,y+t\,<\,y+(1/C)s^p\,\leqslant\,\psi( \psi^{-1}(|y|)+s ).
\]
This shows that $(\psi^{-1}(|y|)+s)+i(y+t) \in \bldblkdm$, for $y, t \geqslant 0$. In
view of our remark on the symmetry of $\bldblkdm$, this completes the proof of
part~\eqref{Cnclsn:ess_planar_incl}.
	
\smallskip 
	
For any $x+iy$ as in the previous paragraphs, $\psi^{-1}(|y|) + o > B \geqslant x$. The first inequality
follows from our choice of $B$. This proves part~\eqref{Cnclsn:ess_planar_x-left}.
	
\smallskip 

Define $K \defeq \{ z\in \C \mid o \leqslant \rprt(z)\leqslant A, \ |\iprt(z)| \leqslant \psi(\rprt(z) - o) \}$.
Write $\bldblkdm_{B} \defeq \{z \in \bldblkdm \mid \rprt(z) \leqslant B\}$. 
For any $x + iy\in \bldblkdm_{B}$:
\[
  o\,\leqslant\,o + \psi^{-1}(|y|)\,<\,o + x\,\leqslant\,o + B\,<\,2o\,<\,A.
\]
Furthermore, $|y| = \psi\big( (\psi^{-1}(|y|) + o) - o \big)$, whence $o + ( \psi^{-1}(|y|)+iy ) \in K$. Clearly, $K$
intersects $\{z \in \C \mid \rprt(z) = A\}$ and $K\setminus \{z \in \C \mid \rprt(z) =A\} \varsubsetneq \bldblkdm$. This
proves part~\eqref{Cnclsn:ess_planar_o-inside}. 
	
\smallskip
	
Finally, for $x+iy \in \bldblkdm_{B}$, $\delta_{\bldblkdm}(x+iy) \leqslant
| (\psi^{-1}(|y|)+iy)-(x+iy) | = |\psi^{-1}(|y|)-x|$ because $\psi^{-1}(|y|)+iy \in \bdy\bldblkdm$. This
proves part~\eqref{Cnclsn:ess_planar_dist} and completes the proof.	
\end{proof}

The next lemma is essentially a parametrized version of the one above. It is related to embedding the model
region $\cspdmaltalt$ into a caltrop within a spike (see Section~\ref{S:caltrops} to recall
terminology),
as we shall see in Section~\ref{S:visibility-caltrops}. A note about our notation: we shall abbreviate
$(z_1,\dots, z_{n-1}, z_n)\in \C^n$ as $(z', z_n)$.

\begin{lemma}\label{L:ess_bulk_lmm}
	Let $\psi:[0, A] \to [0,+\infty)$ be as in Lemma~\ref{L:ess_planar_lmm}. Let
	\[
	  D \defeq \big\{ z \in \C^n \mid 0 < \rprt(z_n) < A, \, \iprt(z_n)^2 + \|z'\|^2 < \psi( \rprt(z_n) )^2
	  \big\}.
	\]
	Let $w' \in \C^{n-1}$ and let
	\[
	  \bldblkdm_{w'} \defeq \pi_n \big[ \big( (w',0) + \{0_{n-1}\} \times \C \big) \cap D \big].
	\]
	Write $\alpha = 1/(p-1)$. Then there exist constants $a,h,B >  0$ and 
	a compact subset $K$ of $\{z \in \C^n \mid \rprt(z_n) \leqslant A\}$ that intersects
	$\{ z \in \C^n \mid \rprt(z_n) = A \}$ and so that $K \setminus \{ z \in \C^n \mid \rprt(z_n)
	= A \} \varsubsetneq D$, such that for every $w' \in \C^{n-1}$ with
	$\|w'\| < \psi( B/2 )$, and every $\zeta
	\in \bldblkdm_{w'}$ with $\rprt(\zeta) \leqslant B$,
	\vspace{1mm} 
	
	\noindent{
	\begin{enumerate}
		\item \label{Cnclsn:par_inc_md_inc} $\big( \psi^{-1}( S(\zeta, w')\,) + i\iprt(\zeta)
		\big)+\cspdmaltalt \subseteq \bldblkdm_{w'}$;
		\vspace{2mm}
		
		\item \label{Cnclsn:par_inc_rch_byd_rprt} $\psi^{-1}( S(\zeta, w')\,) + o >
		\rprt(\zeta)$;
		\vspace{2mm}
		
		\item \label{Cnclsn:par_inc_rch_into_cpt} $\big( \psi^{-1}( S(\zeta, w')\,) +
		i\iprt(\zeta) \big)+o \in \pi_n\big[ \big( (w',0) + \{0_{n-1}\} \times \C \big) \cap K \big]$;
		\vspace{2mm}
		
		\item \label{Cnclsn:par_inc_dist_est} $\delta_D( (w',\zeta) ) 
		\leqslant |\rprt(\zeta) - \psi^{-1}\big( S(\zeta, w')\,\big)|$	;	 
	\end{enumerate}
	where $S(\zeta, w')\defeq \sqrt{ \iprt(\zeta)^2+\|w'\|^2 }$ and
	$o$ is the point in $\cspdmaltalt$ given by \eqref{E:base-pt_o}.}
\end{lemma}

\begin{remark}
The following expression for $\bldblkdm_{w'}$ can easily be obtained:
\[
  \big\{ \zeta \in \C \mid \psi^{-1}(\|w'\|) < \rprt(\zeta) < A, \, \iprt(\zeta)^2 + \|w'\|^2 <  \psi(\rprt(\zeta))^2 \big\}.  
\]
We see that $\bldblkdm_{w'} \neq \varnothing$ if and only if $\|w'\| < \psi(A)$. 
In particular, the sets  $\bldblkdm_{w'}$ appearing in the conclusions of the above lemma
are non-empty.
We also note that $\bldblkdm_{0_{n-1}}$ is
precisely the $\bldblkdm$ of the last lemma. We shall take
the parameters $a,h$ and $B$,
whose existence is asserted above, to be precisely the parameters obtained from
the domain $\bldblkdm = \bldblkdm_{0_{n-1}}$ using Lemma~\ref{L:ess_planar_lmm} above.
\end{remark}
		
\begin{proof}
For simplicity of notation, we shall write $c \defeq 1/C$.
The $w' = 0_{n-1}$ case is precisely the content of Lemma~\ref{L:ess_planar_lmm}.
Let $a, h$ and $B$ be as given by Lemma~\ref{L:ess_planar_lmm}. We extract from the proof
of Lemma~\ref{L:ess_planar_lmm} a couple of simple 
facts that follow from this choice of parameters, and which we shall need in this proof:
\begin{align}
  s+it\in \cspdmaltalt \Rightarrow \ 
  	B + s &< 3A/4 < A \;  \text{and} \;  |t| < cs^p; \label{E:A-B_control} \\
	o &> B. \label{E:A-B_compare}
\end{align}

We now consider the case $w' \neq 0_{n-1}$. Fix a point $\zeta\in \bldblkdm_{w'}$,
and let $\rprt(\zeta) \leqslant B$. That there \emph{is} such a point follows from our
bound on $\|w'\|$. An arbitrary element of $( \psi^{-1}(S(\zeta, w'))
+ i\iprt(\zeta) ) + \cspdmaltalt$ is of the form
\[
  \big( \psi^{-1}(S(\zeta, w')) + s\big) + i\big( \iprt(\zeta) + t \big),
\] 
where $s+it \in \cspdmaltalt$. Such a point belongs to $\bldblkdm_{w'}$ if and only if
\begin{enumerate}
	\item[$(a)$] $\psi^{-1}(S(\zeta, w')) + s < A$, and
	\item[$(b)$] $\|w'\|^2 + (\iprt(\zeta) + t )^2 < \big( \psi( \psi^{-1}(S(\zeta, w')) + s) \big)^2$.
\end{enumerate}
By symmetry, we only need to deal with $\iprt(\zeta) \geqslant 0$. We have
$\iprt(\zeta)^2 + \|w'\|^2 < \psi(\rprt(\zeta))^2$.
As $\rprt(\zeta) \leqslant B$, we have $\iprt(\zeta)^2 + \|w'\|^2 < \psi(B)^2$. 
Therefore
\[
  \psi^{-1}(S(\zeta, w')) + s\,<\,B + s\,<\,A.
\]
The last inequality follows from \eqref{E:A-B_control}. This verifies $(a)$ above. We now verify $(b)$. We have
\[
  \psi( \psi^{-1}(S(\zeta, w')) + s) \geqslant S(\zeta, w') + cs^p,
\]
by an application of Lemma~\ref{L:elem_fn_ineq}. Hence
\[
  \big( \psi( \psi^{-1}(S(\zeta, w')) + s) \big)^2 - \iprt(\zeta)^2 - \|w'\|^2 
  \geqslant 2cs^p S(\zeta, w') + c^2 s^{2p}.
\]
So $(b)$ will follow if we can show that
$2 \iprt(\zeta) t + t^2 < 2cs^p S(\zeta, w') + c^2 s^{2p}$.
But this last inequality is obvious in view of \eqref{E:A-B_control}. 
Thus, $(b)$ is proved, and with it, part~\eqref{Cnclsn:par_inc_md_inc}. 

\smallskip

We note that 
\[
  \psi^{-1}\big( S(\zeta, w') \big) + o\,\geqslant\,\psi^{-1}( \iprt(\zeta) ) + o\,>\,B\,\geqslant\,\rprt(\zeta).
\]
The second inequality above follows from \eqref{E:A-B_compare}. This
proves part \eqref{Cnclsn:par_inc_rch_byd_rprt}.

\smallskip

Let $K \defeq \{(w',\zeta) \in \C^n \mid o \leqslant \rprt(\zeta) \leqslant A, \, S(\zeta, w')
\leqslant \psi(\rprt(\zeta)-o) \}$. Clearly, $K$ is a compact subset of $\{z \in \C^n \mid \rprt(z_n)
\leqslant A\}$ that intersects $\{ z \in \C^n \mid \rprt(z_n) = A \}$, and
$K \setminus \{z \in \C^n \mid \rprt(z_n)=A\} \varsubsetneq D$. Fix $w'$ such that
$\|w'\| < \psi(B/2)$.  Consider a point $\zeta \in \bldblkdm_{w'}$ such that
$\rprt(\zeta) \leqslant B$. 
If we write
\[
  \eta \defeq \big( \psi^{-1} ( S(\zeta, w')) + i\iprt(\zeta) \big) + o,
\]
then we have
\[
  \psi^{-1}(\|w'\|) + o\,\leqslant\,\rprt(\eta)\,<\,\rprt(\zeta) + o\,\leqslant\,B + o\,<\,A. 
\]  
This last inequality follows from \eqref{E:A-B_control} (since $o\in \cspdmaltalt$). 
Furthermore, $S(\eta, w') = S(\zeta, w') = \psi(\rprt(\eta) - o)$. Thus,
$\eta\in \pi_n\big[ \big( (w',0) + \{0_{n-1}\} \times \C \big) \cap K \big]$, which
establishes part \eqref{Cnclsn:par_inc_rch_into_cpt}.

\smallskip

As for part \eqref{Cnclsn:par_inc_dist_est}, if $(w',\zeta)$ is as in the last paragraph,
then $\psi^{-1}(S(\zeta, w'))+i\iprt(\zeta) \in \bdy\bldblkdm_{w'}$.
Therefore $\dtb{D}((w',\zeta)) \leqslant \distance( \zeta, \C \setminus \bldblkdm_{w'})
\leqslant |\rprt(\zeta)- \psi^{-1}(S(\zeta, w'))|$. 
\end{proof}
\smallskip

\section{Caltrops are visibility domains with respect to the Kobayashi metric}\label{S:visibility-caltrops}
This section is devoted to the proof of Theorem~\ref{T:visibility-caltrops}. Our proof
will rely on Theorem~\ref{T:gen_visibility-lemma}. Recall that, in
the discussion related to this theorem, we had mentioned that the utility of 
Theorem~\ref{T:gen_visibility-lemma} lies in that it allows one to identify visibility domains
that \emph{do not} possess the Goldilocks property. The concluding paragraphs of this section
bear this fact out: we shall show that caltrops are not Goldilocks domains.

\smallskip

We shall need the following basic

\begin{lemma}\label{L:psi-prime}
	Let $\psi:[0, A] \to [0,+\infty)$ denote one of the functions
	$\psi_j$ occurring in Definition~\ref{D:caltrop}. Then $\psi$ is
	differentiable at $0$ and $\psi'$ is continuous on $[0, A)$, whence
	$\lim_{x \to 0^+}\psi'(x) = 0$.
\end{lemma}
\begin{proof}
That $\psi'(0)$ exists and equals $0$ follows simply from the bounds on $\psi$. Hence,
$\psi'$ extends to a function on $[0, A)$. The nature of the discontinuities of the derivative
of a univariate function
is such that, since $\psi'$ is increasing on $(0, A)$, it cannot have a discontinuity at $0$.
\end{proof}

\begin{proof}[The proof of Theorem~\ref{T:visibility-caltrops}]
Since we will need Theorem~\ref{T:gen_visibility-lemma} to show that a caltrop 
$\OM\subset \C^n$ is a visibility domain
with respect to the Kobayashi distance, we will require two different types of estimates. We shall therefore
divide our proof into several steps. We begin with the following preliminary remark:
if $F$ and $G$ are two non-negative functions that depend on several parameters, then
we shall write $F \gtrsim G$ to mean that there exists some constant $C$ that is independent of
those parameters such that $G \leqslant C\bcdot F$. The expression $F \thickapprox G$
would mean that  $F\gtrsim G$ and $G\gtrsim F$.

\medskip

\noindent{\textbf{Step~1.} \emph{A lower bound for $\dkoba_{\OM}(w, \bcdot)$ for $w$ contained
in a spike}}
\vspace{0.5mm}

\noindent{Given the set of exceptional points $\{q_1,\dots, q_N\}\subset \bdy{\OM}$, 
\textbf{fix} an exceptional point $q_{j^*}$.
Let $p_{j^*}\in (1, 3/2)$,
$\uni^{(j^*)}\in U(n)$ and $\psi_{j^*} : [0, A_{j^*}]\to [0, +\infty)$ be the data associated to this exceptional point
given by Definition~\ref{D:caltrop}. Since $\kappa_{\OM}$ is invariant under biholomorphisms of $\OM$, and since
the the unitary transformations $\uni^{(j)} = \uni_j^{\prime}$\,---\,where $\uni_j$, $j = 1,\dots, N$, are the
holomorphic maps occurring in Definition~\ref{D:caltrop}\,---\,preserve the (Euclidean) norms of vectors,
we shall, for simplicity of notation, drop the sub/superscript ``$j^*\,$'' from the 
above-mentioned data and assume without loss of generality that
\[
  \OM\cap V_{j^*}\,=\big\{ z \in \C^n \mid 0 < \rprt(z_n) < A, \, \iprt(z_n)^2 + \|z'\|^2 < \psi( \rprt(z_n) )^2
	  \big\}
\]
(so, in the notation just explained, $q_{j^*} = q = 0$).}

\smallskip

We shall now construct a negative plurisubharmonic function on $\OM$ that has an explicit form on
a substantial portion of $\OM\cap V_{j^*}$. This will allow us to use Result~\ref{R:koba_metric_lower}
to obtain a lower bound on $\dkoba_{\OM}(w, \bcdot)$ on a portion of $\OM\cap V_{j^*}$. That there exist such
functions does not follow immediately from the existing theory owing to the presence of singularities in $\bdy{\OM}$. We shall
thus construct a function with the desired properties from basic principles. To this end,
let us write
\[
  \rho(z) \defeq \iprt(z_n)^2 + \|z'\|^2 - \psi( \rprt(z_n) )^2 \quad \forall z\in \OM\cap V_{j^*}.
\]
By the Levi-form calculation in Lemma~\ref{L:leviform}, by Lemma~\ref{L:psi-prime}, and
owing to the properties of $\psi$, we see that there exists a constant $A' \in (0, A]$ such that
\begin{equation}\label{E:levi_lo-bound_spike}
  \levi(\rho)(z; v) \geqslant \|v'\|^2 + 4^{-1}|v_n|^2 \quad
  \forall (z, v)\in (\OM\cap V_{j^*})\times \C^n : 0 < \rprt(z_n) < A'.
\end{equation}

\smallskip

Let $U^{(j)}$, $1\leqslant j\leqslant 4$, be connected open neighbourhoods of $0$ (which represents
$q_{j^*}$ in our present coordinates) such that
\begin{itemize}
	\item $U^{(1)} \Subset U^{(2)} \Subset U^{(3)} \Subset U^{(4)}$; and
	\item $U^{(j)}\cap \OM = \{z\in \OM\cap V_{j^*} \mid 0 < \|z\| < jA'/4\}$, $1\leqslant j\leqslant 4$.
\end{itemize}
Let $\chi_1 \in \smoo^{\infty}(\C^n)$ be such that $\chi_1 : \C^n \to [0, 1]$ and satisfies
\[
  \left.\chi_1\right|_{U^{(1)}} \equiv 0, \quad \text{and}
  \quad  \left.\chi_1\right|_{\C^n\setminus U^{(2)}} \equiv 1.
\]
Let $\phi$ be a smooth, nondecreasing convex function on $[0, +\infty)$ satisfying $\phi(x) = 0$
for each $x\in [0, (A')^2/16]$ that grows very gradually in $((A')^2/16, (A')^2/4]$ and
very rapidly in $[9(A')^2/16, +\infty)$ in a manner that
we shall specify presently. Set $M_{\phi} \defeq \sup_{z\in \OM}\phi(\|z\|^2)$ and
write
\[
  \Phi(z) \defeq \phi(\|z\|^2) - M_{\phi} \quad \forall z\in \OM.
\]
Clearly, $\Phi$ is plurisubharmonic. We compute:
\begin{align*}
  \levi(\rho + \chi_1\Phi)(z; v) = \ &\levi(\rho)(z; v) + \chi_1(z)\levi(\Phi)(z; v) \\
  			& \ + 2\rprt\Big[\sum_{j, k = 1}^n\pdshort{j}{{}}\chi_1\,\pdshort{\overline{k}}{{}}\Phi(z)v_j\overline{v}_k\Big]
			+ \Phi(z)\levi(\chi_1)(z; v) \\
  \geqslant & \ \|v'\|^2 + 4^{-1}|v_n|^2
  			- 2\sum_{j, k = 1}^n\big|\pdshort{j}{{}}\chi_1\,\pdshort{\overline{k}}{{}}\Phi(z)\big|\,|v_j|\,|\overline{v}_k| \\
		& \ - |\Phi(z)|\,|\levi(\chi_1)(z; v)| \quad \forall (z, v) \in \big((U^{(2)}\!\setminus\!U^{(1)})\cap \OM\big)\times \C^n.
\end{align*}
We can drop the term $\chi_1(z)\levi(\Phi)(z; v)$ altogether from the right-hand side of the above
inequality since it is non-negative.  We now state the first of the properties of $\phi$
alluded to above: $\phi$ grows so slowly in the interval
$((A')^2/16, (A')^2/4]$ that
\begin{equation}\label{E:leviform_inner-ring}
  \levi(\rho + \chi_1\Phi)(z; v) \geq 2^{-1}\|v'\|^2 + 8^{-1}|v_n|^2
  \quad \forall (z, v) \in \big((U^{(2)}\!\setminus\!U^{(1)})\cap \OM\big)\times \C^n.
\end{equation}
Now pick $\chi_2\in \smoo_{c}^{\infty}(\C^n)$ such that $\chi_2 : \C^n\to [0,1]$ and satisfies
\[
 \left.\chi_2\right|_{U^{(3)}} \equiv 1, \quad \text{and}
 \quad  \left.\chi_2\right|_{\C^n\setminus U^{(4)}} \equiv 0.
\]
A Levi-form calculation very similar to the one above gives us
\begin{align*}
  \levi(\chi_2\rho + \Phi)(z; v) \geqslant \ &\phi'( \|z\|^2 )\|v\|^2 + \phi''( \|z\|^2 )\left|\langle z, v\rangle\right|^2
  				-  2\sum_{j, k = 1}^n\big|\pdshort{j}{{}}\chi_2\,\pdshort{\overline{k}}{{}}
  				   \rho(z)\big|\,|v_j|\,|\overline{v}_k| \\
				& \ - |\rho(z)|\,|\levi(\chi_2)(z; v)|
				\quad \forall (z, v) \in \big((U^{(4)}\!\setminus\!U^{(3)})\cap \OM\big)\times \C^n.
\end{align*}
The final condition we require on $\phi$ is that $\phi'$ becomes so large on $[9(A')^2/16, +\infty)$ that
we can find a positive constant $c > 0$ so that
\begin{equation}\label{E:leviform_outer-ring}
  \levi(\chi_2\rho + \Phi)(z; v) \geqslant c\|v\|^2
  \quad \forall (z, v) \in \big((U^{(4)}\!\setminus\!U^{(3)})\cap \OM\big)\times \C^n.
\end{equation}
Finally, let us write $u(z) \defeq \chi_1\Phi(z) + \chi_2\rho(z)$ for each $z\in \OM$. Recall that $\Phi$
is plurisubharmonic (which is used
in the calculations above). By this fact, and:
\begin{itemize}
	\item by the maximum principle (applied to $\Phi$), and by the choice of the functions
	$\chi_j$, $j = 1, 2$, we see that $u < 0$ on $\OM$.
	\item from the choice of the functions $\chi_j$, $j = 1, 2$, and from the inequalities
	\eqref{E:leviform_inner-ring} and \eqref{E:leviform_outer-ring}, it follows that $u$ is plurisubharmonic
	on $\OM$.
\end{itemize}

\smallskip

By the rotational symmetry of the spike $\OM\cap V_{j^*}$, it follows that for any $w\in \OM\cap V_{j^*}$
with $\rprt(w_n)$ sufficiently small, we have
\[
  \delta_{\OM}(w) = {\rm dist}\big( \rprt(w_n) + iS(w), \mathsf{graph}(\psi) \big).
\]
Here, $S(w)$ is our abbreviation for $\sqrt{\iprt(w_n)^2 + \|w'\|^2}$. From this last observation and
elementary calculus, it follows that for such a $w$, if $\xi^w\in \bdy{\OM}$ is a point such that
$\|w - \xi^w\| = \delta_{\OM}(w)$, then
\[
  \rprt(w_n) - \rprt\big( \pi_n(\xi^w) \big) \in \big( 0, \psi(\rprt(w_n))\psi'(\rprt(w_n)) \big).
\]
Thus, it follows from Lemma~\ref{L:psi-prime} and a few elementary estimates
that (we write $\xi^w_n \defeq \pi_n(\xi^w)$ henceforth)
\[
  \frac{\delta_{\OM}(w)}{\psi(\rprt(w_n)) - S(w)}
  =  \frac{\big| \big( \rprt(\xi^w_n) - \rprt(w_n) \big) + i\big( \psi( \rprt(\xi^w_n) ) - S(w) \big) \big|}
  	{\psi(\rprt(w_n)) - S(w)} \to 1 \text{ as } \rprt(w_n) \to 0.
\]
Hence, there exists a constant $A'' > 0$ such that
\begin{align}
  \{z \in \OM\cap V_{j^*} \mid 0 < \rprt(z_n) < A''\} &\subset  \OM\cap U^{(1)} \; \;\text{and}
  \; \;\psi(x) \in (0, 1) \; \; \forall x \in (0, A'');
  \label{E:small_spike} \\
  \frac{\delta_{\OM}(w)}{\psi(\rprt(w_n)) - S(w)} &> \frac{1}{2} \; \; \forall w: \rprt(w_n) \in (0, A'')
  \label{E:delta_estimate}
\end{align}

\smallskip

We now appeal to Result~\ref{R:koba_metric_lower}. Fix a point $w\in \OM\cap V_{j^*}$ such that
$0 < \rprt{w_n} < A''$. Then, there exists a constant $b > 0$ such that
\begin{align}
  \kappa_{\OM}(w; v) &\geqslant b\frac{\|v\|}{|u(w)|^{1/2}} \label{E:koba-metric_delicate}\\
  &= b\,\frac{\|v\|}{\big( \psi(\rprt(w_n)) - S(w) \big)^{1/2}\big( \psi(\rprt(w_n)) + S(w) \big)^{1/2}}
  &&  (\text{by \eqref{E:small_spike}, given def'ns.~of $\chi_1$, $\chi_2$}) \notag \\
  &\geqslant \frac{b}{\sqrt{2}}\,\frac{\|v\|}{\big( \psi(\rprt(w_n)) - S(w) \big)^{1/2}}
  && (\text{by \eqref{E:small_spike} above}) \notag \\
  &\geqslant \frac{b}{2}\,\frac{\|v\|}{\delta_{\OM}(w)^{1/2}}
  && (\text{by \eqref{E:delta_estimate} above}), \notag
\end{align}
and this estimate holds for any arbitrary $w$ as described above\,---\,i.e., $w\in \OM\cap V_{j^*}$ such that
$0 < \rprt{w_n} < A''$.

\medskip

\noindent{\textbf{Step~2.} \emph{An upper bound for $M_{\OM}$}}
\vspace{0.5mm}

\noindent{Since the exceptional point $q_{j^*}$ in Step~1 was arbitrarily chosen, we actually infer the following
from Step~1: there exist a constant $\beta > 0$ and constants $A''_1,\dots, A''_N > 0$ such that
\begin{align}
  \dkoba_{\OM}(w; v) \geqslant \beta\,\frac{\|v\|}{\delta_{\OM}(w)^{1/2}}
  \quad &\forall w\in \OM\cap \uni^{-1}_j\big( \{z \in \C^n \mid \rprt(z_n) < A''_j\} \big) \text{ and} \notag \\
  &\forall v\in \C^n, \label{E:koba_metric_spike}   
\end{align}
where $j = 1,\dots, N$. Now, let us write:
\begin{align*}
  \mani_0 &\defeq \bdy{\OM} \cap \bigcap\nolimits_{1\leqslant j\leqslant N}\!\uni^{-1}_j
  \big( \{z \in \C^n \mid \rprt(z_n) > A''_j/2\} \big), \\
  \mani_1 &\defeq \bdy{\OM} \setminus \bigcup\nolimits_{1\leqslant j\leqslant N}\!\uni^{-1}_j
  \big( \{z \in \C^n \mid \rprt(z_n) < A''_j\} \big).
\end{align*}
It is clear from Definition~\ref{D:caltrop} and from very standard facts about strongly Levi-pseudoconvex
hypersurfaces that the Levi-nondegeneracy condition stated in Result~\ref{R:koba_metric_M-BB}
holds true at every $\xi\in \mani_0$. Thus, it follows from Result~\ref{R:koba_metric_M-BB} that there
exists an $\overline{\OM}$-open neighbourhood $\mathcal{V}$ of $\mani_1$ and a constant $\beta' > 0$
such that
\begin{equation}\label{E:koba_metric_elsewhere}
  \dkoba_{\OM}(w; v) \geqslant \beta'\,\frac{\|v\|}{\delta_{\OM}(w)^{1/2}}
  \quad \forall w\in \mathcal{V}\cap \OM \text{ and } \forall v\in \C^n.
\end{equation}
Now, by definition, the set
\[
 \OM\setminus \left(\mathcal{V} \cup \bigcup\nolimits_{1\leqslant j\leqslant N}\!\uni^{-1}_j
  \big( \{z \in \C^n \mid \rprt(z_n) < A''_j\} \big) \right)
\]
is compact. Thus, in view of \eqref{E:koba_metric_spike} and \eqref{E:koba_metric_elsewhere}, it follows that
\[
  \frac{1}{\dkoba_{\OM}(w; v)} \lesssim\delta_{\OM}(w)^{1/2}
  \quad \forall w\in \OM \text{ and } \forall v\in \C^n : \|v\| = 1.
\]
In particular, $M_{\OM}(r) \lesssim r^{1/2}$.}

\medskip

\noindent{\textbf{Step~3.} \emph{The behaviour of $\koba_{\OM}$}}
\vspace{0.5mm}

\noindent{Let us initially \textbf{fix} an exceptional point in the set $\{q_1,\dots, q_N\}$.
Let $a_j, h_j$ and $B_j$ be the constants given by Lemma~\ref{L:ess_bulk_lmm}
taking $\psi = \psi_{j}$. For simplicity of notation, we shall denote the first two constants
as $a$ and $h$\,---\,with the dependence on $j$ being understood. Consider a point $w$:
\[
  w \in \OM\cap \uni^{-1}_j\big( \{z \in \C^n \mid \rprt(z_n) < B_j/2\} \big),
\]
and write $\uni_j(w) = (\omega', \omega_n)$.
Next, consider the holomorphic map
$\Psi_{j, w} : \cspdmaltalt\to \C^n$ given by
\[
  \Psi_{j, w}(\zeta) \defeq \uni^{-1}_j\big( \omega', \psi^{-1}_j(S(\omega)) + i\iprt(\omega_n)
  + \zeta \big) \quad \forall \zeta\in \cspdmaltalt, \vspace{1mm}
\]
where $S(\omega) \defeq \sqrt{\iprt(\omega_n)^2 + \|\omega'\|^2}$, and
$\cspdmaltalt$ is the domain constructed in Section~\ref{S:comparison_dom}.
The parameters $a$ and $h$ are as just described above. We take $\alpha =  1/(p_j - 1)$.
Note that this map is a $\C$-affine embedding of $\cspdmaltalt$ into $\C^n$. As hinted in
Section~\ref{S:comparison_dom}, we shall show that $\Psi_{j, w}$ embeds $\cspdmaltalt$
in $\OM$\,---\,by which we can estimate $\koba_{\OM}$.}
\smallskip

Observe that $\rprt(\omega_n) < B_j/2$.
Hence, given the set to which $w$ belongs and by Definition~\ref{D:caltrop}, $\|\omega'\| < \psi_j(B_j/2)$.
Thus, it follows from parts~\eqref{Cnclsn:par_inc_md_inc}
and \eqref{Cnclsn:par_inc_rch_byd_rprt} of Lemma~\ref{L:ess_bulk_lmm} that:
\begin{itemize}[leftmargin=27pt]
	\item[$(a)$] With $w$ as chosen above,
	\begin{multline*}
	  \{\omega'\}\times \big( \psi^{-1}_j(S(\omega)) + i\iprt(\omega_n)
	  + \cspdmaltalt \big) \\
	  	\subset \big\{ (z', z_n)\in \C^n \mid \rprt(z_n) \in (0, A_j), \; 
		\iprt(z_n)^2 + \|z'\|^2 < \psi_j\big( \rprt(z_n) \big)^2 \big\}.
	\end{multline*}
	\item[$(b)$] $\rprt(\omega_n) \in (\psi^{-1}_j(S(\omega)),\,\psi^{-1}_j(S(\omega)) + o)$.
\end{itemize}
Here $o\in \cspdmaltalt$ is as provided by \eqref{E:base-pt_o} for the
above-mentioned choice of parameters.  Now, write
\[
  K_{j} \defeq \uni^{-1}_j(K) \quad\text{and}
  \quad z_w \defeq \Psi_{j, w}(o),
\]
where $K$ is the compact set given by Lemma~\ref{L:ess_bulk_lmm} taking
$\psi = \psi_j$ in that lemma. Then, it follows from part~\eqref{Cnclsn:par_inc_rch_into_cpt}
of the latter lemma that
\begin{equation}\label{E:in_K_j}
  z_w \in K_j \quad\forall w\in \OM\cap V_j : 0 < \rprt(\omega_n) < B_j/2.
\end{equation}
From $(a)$ we see that $\Psi_{j, w}(\cspdmaltalt) \subset \OM$. By definition,
$\Psi_{j, w}(-\psi_j^{-1}(S(\omega)) + \rprt(\omega_n)) = w$.
Thus, as holomorphic maps
are contractive relative to the Kobayashi distance, we have
\[
  \koba_{\OM}(z_w, w) \leqslant \koba_{\cspdmaltalt}(o, -\psi_j^{-1}(S(\omega)) + \rprt(\omega_n)).
\] 
In view of $(b)$, part~\eqref{Cnclsn:cusp_domain_koba_est} 
of Proposition~\ref{P:cusp_domain_props} gives us\,---\,taking $x_0 = o$ in that proposition\,---\,the
estimate
\[
  \koba_{\OM}(z_w, w) \leqslant C^{(j)} + \frac{\pi}{4h}| \rprt(\omega_n) - \psi^{-1}_j(S(\omega)) |^{-(p_j - 1)},
\]
for some constant $C^{(j)} > 0$.  Since the maps $\uni_j$ and $\uni^{-1}_j$ preserve Euclidean distances,
the above inequality together with part~\eqref{Cnclsn:par_inc_dist_est} of Lemma~\ref{L:ess_bulk_lmm}
gives us the following:
\begin{equation}\label{E:koba-dist_imp-estimate}
  \koba_{\OM}(z_w, w) \leqslant C^{(j)} + \frac{\pi}{4h}\delta_{\OM}(w)^{-(p_j - 1)} \quad
  \forall w\in \OM\cap V_j : 0 < \rprt(\omega_n) < B_j/2.
\end{equation}
Since the exceptional point $q_j$ was chosen arbitrarily in this discussion, the statements
\eqref{E:in_K_j}  and \eqref{E:koba-dist_imp-estimate} hold for each $j = 1,\dots, N$.

\smallskip

Since $\bdy{\OM}$ is of class $\smoo^2$ away from the points $q_1,\dots, q_N$, and $\OM$
is bounded, it is routine to find a compact set $K_0 \subset \OM$ and a constant
$R > 0$ such that for each point
\begin{equation}\label{E:close_to_smooth_point}
  w \in \OM \setminus \left(K_0 \cup \bigcup\nolimits_{1\leqslant j \leqslant N}\!\uni^{-1}_j
  \big( \{z \in \C^n \mid \rprt(z_n) < B_j/2\} \big) \right)
\end{equation}
there exists a point
\[
  \xi^w\in \bdy{\OM} \setminus \bigcup\nolimits_{1\leqslant j \leqslant N}\!
  		\uni^{-1}_j\big( \{z \in \C^n \mid \rprt(z_n) < B_j/4\} \big)
\]
so that, if $\eta^w$ denotes the unit inward-pointing normal vector to $\bdy{\OM}$, then 
\begin{itemize}[leftmargin=22pt]
	\item[$(a')$] $\xi^w + D(R; R)\eta^w \subset \OM$;
	\item[$(b')$] $w$ lies on the line segment joining $\xi^w$ to $\xi^w+ R\eta^w \defines z^w$; and
	\item[$(c')$] $z^w \in K_0$.
\end{itemize}
Thus, for each $w$ as indicated above, there is a unique number $t(w) \in (0, R)$ such that
$\xi^w + t(w)\eta^w = w$.
From this (and the fact that holomorphic maps are contractive relative to the Kobayashi distance) it follows that
\begin{align}
  \koba_{\OM}(z^w, w) \leqslant \koba_{D(R; R)}(0, t(w)) = \ &\frac{1}{2}\log\left(\frac{2 - (t(w)/R)}{t(w)/R}\right)
  \notag \\
  \leqslant \ &\log(\sqrt{2}) + \frac{1}{2}\log\left(\frac{1}{\|\xi^w - w\|}\right)
  \notag \\
  \leqslant \ &\log(\sqrt{2}) + \frac{1}{2}\log\left(\frac{1}{\delta_{\OM}(w)}\right) \notag \\
  &\forall w \text{ satisfying the condition given by \eqref{E:close_to_smooth_point}.}
  \label{E:koba-dist_simple-estimate}
\end{align}
Let us now fix a point $z_0\in \OM$. Write
\begin{align*}
  K^* &\defeq K_0\cup K_1\cup\dots \cup K_N, \\
  C_0 &\defeq \sup\nolimits_{x\in K^*}\koba_{\OM}(z_0, x) +
  		\max\big( \log(\sqrt{2}), C^{(1)},\dots, C^{(N)} \big).
\end{align*}
Then, by the triangle inequality for $\koba_{\OM}$, \eqref{E:in_K_j}, and by the inequalities 
\eqref{E:koba-dist_imp-estimate}  and \eqref{E:koba-dist_simple-estimate} it follows that there
exists a constant $C_1 > 0$ such that
\[
  \koba_{\OM}(z_0, z) \leqslant C_0 + C_1\delta_{\OM}(z)^{-\max_{1\leqslant j\leqslant N}p_j +1}
  \quad \forall z\in \OM.
\]
\smallskip

\noindent{\textbf{Step~4.} \emph{Caltrops are visibility domains with respect to the Kobayashi distance}}
\vspace{0.5mm}

\noindent{Let us write
$p_0 \defeq \max_{1\leqslant j\leqslant N}p_j$. Then, by hypothesis, $p_0 \in (1, 3/2)$.
We shall complete the proof using Theorem~\ref{T:gen_visibility-lemma}. In the notation of that theorem,
we can\,---\,using the conclusion of Step~3\,---\,take $f(r) = C_0 + C_1r^{p_0-1}$.
Thus, using the conclusion of Step~2, we have
\[
  0\,\leqslant\,\frac{M_{\OM}(r)}{r^2}\,f'\!\!\left(\frac{1}{r}\right)\,\lesssim\,\frac{r^{1/2}}{r^2}\bcdot r^{2-p_0}
  = \frac{1}{r^{p_0-(1/2)}}.
\]
As $p_0 < 3/2$, we have
\[
  0 \leqslant \int_0^{r_0} \frac{M_{\OM}(r)}{r^2}\,f'\!\!\left( \frac{1}{r} \right) dr \lesssim
  \int_0^{r_0}\frac{dr}{r^{p_0-(1/2)}}  < \infty
\]
for $r_0$ so small that $(0, r_0)$ is included in the domain of the integrand.
Hence, we conclude from Theorem~\ref{T:gen_visibility-lemma} that $\OM$ is a visibility domain with respect
to the Kobayashi distance.}

\medskip

\noindent{\textbf{Step~5.} \emph{Caltrops are not Goldilocks domains}}
\vspace{0.5mm}

\noindent{We will show that the condition on the growth of the Kobayashi distance that Goldilocks domains must
satisfy fails in a caltrop. To do so, we fix an
exceptional point $q_{j^*}$ and refer the reader to Step~1 for an explanation for why
we can, without loss of generality, take $\OM\cap V_{j^*}$ to be
\begin{equation}\label{E:good_prsntn}
  \OM\cap V_{j^*}\,=\big\{ z \in \C^n \mid 0 < \rprt(z_n) < A, \, \iprt(z_n)^2 + \|z'\|^2 < \psi( \rprt(z_n) )^2
	  \big\}
\end{equation}
(recall that $V_{j^*}$ is the neighbourhood of $q_{j^*}$ given by Definition~\ref{D:caltrop}).
As in Step~1, we drop, for the moment, the sub/superscript ``$j^*$''.}

\smallskip

At this stage, we shall need the following

\begin{lemma}\label{L:koba_lower_bnd}
	Fix an exceptional point $q_{j^*}\in \bdy\OM$ and let $(z_1,\dots, z_n)$ be the system
	of holomorphic coordinates centred at $q_{j^*}$ such that $\OM\cap V_{j^*}$ has
	the form \eqref{E:good_prsntn}. Let $A''$ be as introduced just prior to \eqref{E:small_spike},
	and let $z_0 = (0,\dots, 0, A''/2)$. Then, for any $z\in \OM\cap V_{j^*}$ such that
	$0 <\rprt(z_n) < A''/2$, we have
	\begin{equation}\label{E:koba_lower_bnd}
	  \koba_{\OM}(z_0, z)\,\gtrsim\,\rprt(z_n)^{-(p-1)} - (A''/2)^{-(p-1)}.
	\end{equation}
\end{lemma}

We shall defer the proof of this lemma until the end of this section. Instead, let us use
it to complete this proof. Write $\xpoint \defeq (0,\dots, 0, x)$, $0 < x < A''/2$.
Now, take $z = \xpoint$ in the above lemma to get
\[
  \koba_{\OM}(z_0, \xpoint)\,\gtrsim\,x^{-(p-1)} - (A''/2)^{-(p-1)}.
\]
Now, substitute $\xpoint$ for the $w$
in the statement just prior to \eqref{E:small_spike} in
Step~1 to infer that $\delta_{\OM}(\xpoint) \thickapprox x^p$ for any $x\in (0, A''/2)$. Applying this to
the last estimate, we have
\[
  \koba_{\OM}(z_0, \xpoint)\,\gtrsim\,\delta_{\OM}(\xpoint)^{-1+(1/p)} - (A''/2)^{-(p-1)}.
\]
Since, $p = p_{j^*} > 1$, 
$\delta_{\OM}(\xpoint)^{-1 + (1/p)}/\log\big( 1/\delta_{\OM}(\xpoint) \big) \to +\infty$
as $\xpoint \to q_{j^*}$. Thus
$\koba_{\OM}(z_0, \xpoint)$ \textbf{cannot} satisfy the upper bound \eqref{E:Goldilocks_koba-distance_bound}
for any choice of constants $C, \alpha > 0$ as $\xpoint \to q_{j^*}$. Thus the caltrop $\OM$ is not a
Goldilocks domain.
\end{proof}

We now provide

\begin{proof}[The proof of Lemma~\ref{L:koba_lower_bnd}]
Fix a $z\in \OM\cap V_{j^*}$ with $0 <\rprt(z_n) < A''/2$.
In proving this lemma, we shall use a slightly different 
lower bound for $\dkoba_{\OM}$, which was also derived in Step~1. Let
\[
  \mathscr{C}(z) \defeq \text{the class of all piecewise $\smoo^1$
  paths $\gamma : ([0, 1], 0, 1) \to (\OM, z_0, z)$}.
\]
As discussed at the beginning of Section~\ref{S:prelims_technical}:
\begin{equation}\label{E:integ_dist}
  \koba_{\OM}(z_0, z) =
  \inf_{\gamma\in \mathscr{C}(z)}\int_0^{1}\dkoba_{\OM}(\gamma(t); \gamma'(t))\,dt.
\end{equation}
Pick a $\gamma\in \mathscr{C}(z)$. Since $\rprt(\gamma_n)$ is a continuous function and
$\rprt(\gamma_n)([0, 1]) \supset [\rprt(z_n), A''/2]$, it follows from elementary topological considerations
that there exist numbers $\alpha, \beta\in [0, 1]$ such that
\[
  \rprt(\gamma_n)([\alpha, \beta]) = [\rprt(z_n), A''/2].
\]
Therefore
\begin{align}
  \int_0^{1}\dkoba_{\OM}(\gamma(t); \gamma'(t))\,dt
  &\geqslant \int_{\alpha}^{\beta}\dkoba_{\OM}(\gamma(t); \gamma'(t))\,dt \notag \\
  &\geqslant \int_{\alpha}^{\beta}\frac{b\|\gamma'(t)\|}{|u(\gamma(t))|^{1/2}}\,dt, \label{E:interim_integral}
\end{align}
The second inequality above follows from the estimate \eqref{E:koba-metric_delicate}.

\smallskip

For any point $w\in \OM\cap V_{j^*}$ with $0 < \rprt(w_n) < A''$ we have
\[
  |u(w)| = \psi( \rprt(w_n) )^2 - \|w'\|^2 - \iprt(w_n)^2 \leqslant \psi( \rprt(w_n) )^2 \leqslant
  C^2\rprt(w_n)^{2p}
\]
where $C > 0$ is the constant $C_{j^*}$ mentioned in Definition~\ref{D:caltrop}.
Therefore, \eqref{E:interim_integral} gives us (the last three integrals below are Riemann
integrals; it is not hard to establish that the integrands are Riemann integrable):
\begin{align*}
  \int_0^{1}\dkoba_{\OM}(\gamma(t); \gamma'(t))\,dt
  &\geqslant \frac{b}{C}\int_{\alpha}^{\beta}\frac{|(\rprt(\gamma_n))'(t)|}{\rprt(\gamma_n(t))^p}\,dt \\
  &\geqslant \frac{b}{C}\left|\int_{\alpha}^{\beta}\frac{(\rprt(\gamma_n))'(t)}{\rprt(\gamma_n(t))^p}\,dt\right| \\
  &= \frac{b}{C}\int_{\rprt(z_n)}^{A''/2}\frac{1}{t^p}\,dt\,.
\end{align*}
A few words about the change-of-variables formula that gives the last equality: since $\gamma$ is piecewise
$\smoo^1$, we invoke (a small refinement of) the classical change-of-variables formula on a finite
collection of subintervals that tile $[\alpha, \beta]$.
Recalling that $\gamma$ was chosen arbitrarily from the class $\mathscr{C}(z)$, this last estimate, together
with \eqref{E:integ_dist}, gives us \eqref{E:koba_lower_bnd}.
\end{proof}
\smallskip

\section{Caltrops are taut}
In this section, we shall prove that any caltrop is taut. While this is believable, it takes a
little effort to show owing to the exceptional points in the boundary of a caltrop
$\OM$. At these points, $\bdy\OM$ is not just non-smooth but is not even Lipschitz (were
$\bdy\OM$ Lipschitz, tautness would have followed from a result of Kerzman--Rosay 
\cite{Kerzman_Rosay}). We first need the following standard result.

\begin{lemma}\label{L:koba_complete}
	Let $\OM$ be a bounded domain in $\C^n$ and suppose that for each $z_0\in \OM$
	and $\xi\in \bdy\OM$, we have
	\begin{equation}\label{E:distance_blow-up}
	  \lim_{\OM\ni w\to \xi}\koba_{\OM}(z_0, w) = +\infty.
	\end{equation}
	Then the metric space $(\OM, \koba_{\OM})$ is (Cauchy) complete.
\end{lemma}
\noindent{The proof of this lemma involves very standard arguments. We use the conclusion of
Result~\ref{R:koba_facts_misc} and the fact that the metric-topology on $\OM$ induced
by $\koba_{\OM}$ coincides with its standard topology. We skip the routine details.}

\smallskip

With this, we are in a position to prove

\begin{theorem}\label{T:caltrops_taut}
	Caltrops are complete relative to the Kobayashi distance. In particular, they are taut.
\end{theorem}
\begin{proof}
Let $\OM$ be a caltrop, and let $\{q_1,\dots, q_N\}\subset \bdy{\OM}$ be the set of
exceptional boundary points.
Fix a point $z_0\in \OM$ and a point $\xi\in \bdy\OM$. First,
we consider the case where $\xi\in \bdy{\OM}\setminus \{q_1,\dots, q_N\}$.
Pick a point $\eta\in \bdy{\OM}\setminus (\{q_1,\dots, q_N\}\cup \{\xi\})$.
Thus, $\bdy\OM$ is
strongly Levi-pseudoconvex around $\xi$ and $\eta$. We now appeal to
Result~\ref{R:koba_dist_lower_bd}: let $V_{\xi}$ and $V_{\eta}$ be the neighbourhoods
and let $C > 0$ be the constant given by this result. Let $b_{\eta}$ be some point in $\OM\cap V_{\eta}$.
Consider any sequence $(w_{\nu})_{\nu\geqslant 1}\subset \OM$
such that $w_{\nu}\to \xi$. Without loss of generality, we may assume that this
sequence is contained in $\OM\cap V_{\xi}$. Then, Result~\ref{R:koba_dist_lower_bd}
tells us that
\begin{align*}
  \koba_{\OM}(z_0, w_{\nu}) &\geqslant \koba_{\OM}(w_{\nu}, b_{\eta}) - \koba_{\OM}(b_{\eta}, z_0) \\
  							&\geqslant 2^{-1}\log\frac{1}{\delta_{\OM}(w_{\nu})}
	  							+ 2^{-1}\log\frac{1}{\delta_{\OM}(b_{\eta})}
	  							- \koba_{\OM}(b_{\eta}, z_0) - C \\
							&\to +\infty \text{ as $\nu\to \infty$}.
\end{align*}
As $\xi$ was arbitrarily chosen from $\bdy{\OM}\setminus \{q_1,\dots, q_N\}$, the above
establishes \eqref{E:distance_blow-up} for any non-exceptional boundary point.

\smallskip

Now, let $\xi$ be an exceptional boundary point. As in the statement of Lemma~\ref{L:koba_lower_bnd},
call this point $q_{j^*}$ and let $(z_1,\dots, z_n)$ be the system of holomorphic coordinates described
in this lemma.  Consider any sequence $(w_{\nu})_{\nu\geqslant 1}\subset \OM$
such that $w_{\nu}\to q_{j^*}$.  
Without loss of generality, we may assume that this
sequence is contained in $\OM\cap V_{j^*}$. Now, Lemma~\ref{L:koba_lower_bnd} is stated
keeping in mind a specific assumption about $\OM\cap V_{j^*}$ (stated just prior to it). Here too
we may without loss of generality assume that $\OM\cap V_{j^*}$ \textbf{is} the set given
by \eqref{E:good_prsntn}.
This is because the coordinates $(z_1,\dots, z_n)$ are
given by a biholomorphism defined on all of $\OM$ (indeed, on all of $\C^n$).
With this assumption, we shall identify $z_n(w_{\nu})$ and $\pi_n(w_{\nu}) \defines w_{\nu, n}$.
Then (with this assumption) we have
\begin{equation}\label{E:limit_n-th_coord}
  \rprt(w_{\nu, n}) \to 0 \text{ as $\nu \to \infty$}.
\end{equation}
Let $A''$ the constant given by
Lemma~\ref{L:koba_lower_bnd}, and let us denote the point $z_0$ mentioned in this
lemma by $\zeta_{j^*}$ (to avoid confusion with the $z_0$ fixed above). Then, this lemma tells us that
\begin{align*}
  \koba_{\OM}(\zeta_{j^*}, w_{\nu}) &\gtrsim\,\rprt(w_{\nu, n})^{-(p-1)} - (A''/2)^{-(p-1)}
								\to +\infty \text{ as $\nu\to \infty$}.
\end{align*}
The last statement follows from \eqref{E:limit_n-th_coord}. Owing to the triangle inequality
for $\koba_{\OM}$, the above suffices to establish \eqref{E:distance_blow-up} for
$\xi = q_{j^*}$. Together with the conclusion of the previous paragraph, we
conclude\,---\,using Lemma~\ref{L:koba_complete}\,---\,that $(\OM, \koba_{\OM})$
is (Cauchy) complete.

\smallskip

As $(\OM, \koba_{\OM})$ is complete, it follows from a result of Kiernan \cite{Kiernan} that
$\OM$ is taut.
\end{proof} 							
\smallskip

\section{Wolff--Denjoy theorems}\label{S:Wolff_Denjoy}
We now have all the tools needed to prove the two Wolff--Denjoy-type theorems, and a
corollary, stated in Section~\ref{S:intro}. We reiterate that the key heuristic in the proof of
Theorem~\ref{T:Wolff_Denjoy_gen} is as stated in the second paragraph following the
statement of Theorem~\ref{T:Wolff_Denjoy_gen}. Since that heuristic is entirely a
consequence of visibility, large parts of the proof below will be similar to the proof of
\cite[Theorem~1.10]{Bharali_Zimmer} for Goldilocks domains, which also relies on
this heuristic. The supporting lemmas/theorems to the proof below are those that
show that the quantitative conditions defining a Goldilocks domain are not needed.

\smallskip

The proof of Theorem~\ref{T:Wolff_Denjoy_gen} involves the analysis of two separate cases,
one of which is rather technical. This is because we do \textbf{not} assume that
$(\OM, \koba_{\OM})$ is Cauchy complete in Theorem~\ref{T:Wolff_Denjoy_gen}\,---\,to do
so would be too restrictive. To illustrate: it is not
known whether, for a weakly pseudoconvex domain $\OM\Subset \C^n$, $n\geq 3$,
$(\OM, \koba_{\OM})$ is Cauchy complete (that such a domain is a visibility domain 
follows from \cite[Theorem~1.4]{Bharali_Zimmer}). In contrast, tautness
is much simpler to determine in practice, and suffices for the conclusion of
Theorem~\ref{T:Wolff_Denjoy_gen}. With these words, we give

\begin{proof}[The proof of Theorem~\ref{T:Wolff_Denjoy_gen}]	
Since $\Omega$ is taut, it follows from a result by Abate \cite[Theorem~2.4.3]{Abate_ItThTautMan} that
either the set $\{F^{\nu} \mid \nu\in \posint\}$ is relatively compact in $\hol(\OM;\OM)$ or
$(F^{\nu})_{\nu \geqslant 1}$ is compactly divergent on $\OM$. In the former case, clearly, for each $z \in \OM$,
the orbit $\{F^{\nu}(z) \mid \nu \in \posint\}$ is relatively compact in $\OM$.

\smallskip

Hence we now suppose that $(F^{\nu})_{\nu \geqslant 1}$ is compactly divergent. 
By Montel's theorem, there exist subsequences of $(F^{\nu})_{\nu \geqslant 1}$ that converge uniformly
on compact
subsets of $\OM$ to $\bdy\OM$-valued holomorphic maps.
By Theorem~\ref{T:bdy-vld_lmts_cnstnt}, the latter maps are constant maps. Thus, we shall identify the set
\[
  \Gamma \defeq \overline{ \{F^{\nu} \mid \nu \in \posint\}}^{\,{\rm compact-open}}
  \setminus \{F^{\nu} \mid \nu \in \posint\}
\]
as a set of points in $\partial\Omega$. In a similar vein, we shall refer to the constant maps
$\textsf{const}_{p}$, where $p\in \bdy\OM$, simply as $p$.
Our goal is to show that $\Gamma$ is a single point.
We assume, to get a contradiction, that $\Gamma$ contains at least two
points. We divide our discussion into two cases:

\smallskip

\noindent{\textbf{Case~1.}} We first consider the case in which for some (and hence any) $o \in \OM$, 
\[ 
  \limsup_{\nu \to \infty} \koba_{\OM}(F^{\nu}(o),o) = \infty.
\] 
We ought to mention here that (as implied by the discussion following the statement of
Theorem~\ref{T:Wolff_Denjoy_gen}) the essence of the argument under the heading ``Case~1'' in the proof of
\cite[Theorem~1.10]{Bharali_Zimmer} applies in the present, more general, setting. The chief differences
are that:
\begin{itemize}
	\item the lemmas/propositions supporting the two arguments differ; and
	\item since the Goldilocks condition in \cite{Bharali_Zimmer} involves an upper bound
	on $\koba_{\OM}$, certain inequalities and observations (e.g., \eqref{E:minus-l_seq} below) needed
	no argument in the latter work, but for which we provide explanations (when needed) here.
\end{itemize}     
In this case we can find a strictly increasing sequence $(\nu_i)_{i \geqslant 1}\subset \posint$ such that for every
$i \in \posint$ and every $k \leqslant \nu_i$, $\koba_{\OM}(F^k(o),o) \leqslant \koba_{\OM}(F^{\nu_i}(o),o)$.
By passing to a subsequence and relabelling, if necessary, we may assume that 
$F^{\nu_i}\to \xi$ uniformly on compact subsets of $\OM$ for some
$\xi \in \bdy \OM$. By assumption, there is a subsequence 
$(F^{\mu_j})_{j \geqslant 1}$ that converges uniformly on compact subsets to
$\eta$, where $\eta \in \bdy \OM$ and $\eta \neq \xi$.
By Proposition~\ref{P:iteration_seqs},
it cannot be the case that
\[
  \limsup_{j \to \infty} \koba_{\OM}(F^{\mu_j}(o),o) = \infty,
\]
since $\eta \neq \xi$. Therefore, $\limsup_{j \to \infty} \koba_{\OM}(F^{\mu_j}(o),o) < \infty$.
Hence, by the triangle inequality:
\begin{align}
  \limsup_{i \to \infty} \limsup_{j \to \infty} \koba_{\OM}\big( F^{\nu_i}(o), F^{\mu_j}(o) \big) 
    &\geqslant
    \limsup_{i \to \infty} \limsup_{j \to \infty} \big( \koba_{\OM}\big( F^{\nu_i}(o),o \big) - \koba_{\OM} \big(
    F^{\mu_j}(o),o \big) \big) \notag \\
  &\geqslant \limsup_{i \to \infty} \Big[ \koba_{\OM}\big( F^{\nu_i}(o),o \big) -
     \limsup_{j \to \infty} \koba_{\OM}\big( F^{\mu_j}(o),o \big) \Big] \notag \\
  &= \infty. \label{E:plus_infty_outcome}
\end{align}

\smallskip

Fix an $\ell\in \posint$. When we apply Theorem~\ref{T:bdy-vld_lmts_cnstnt} to any subsequence 
of $(F^{\mu_j-\ell})_{j \geqslant 1}$ that converges uniformly on compact subsets of $\OM$, we get\,---\,since
any such subsequence converges to $\eta$ on the compact $K_{\ell} \defeq \{F^{\ell}(o)\}$\,---\,that
\begin{equation}\label{E:minus-l_seq}
  (F^{\mu_j-\ell})_{j \geqslant 1} \text{ converges uniformly on compact subsets to $\eta$}.
\end{equation} 
Let us define
\[
  M_{\ell} \defeq \limsup_{j \to \infty} \koba_{\OM}(F^{\mu_j-\ell}(o),o).
\]
We claim that 
\[
  \limsup_{\ell \to \infty}M_{\ell} < \infty.
\]
Suppose not. Then there is a strictly increasing sequence $(\ell_k)_{k \geqslant 1} \subset \posint$ such that for
each $k \in \posint$, $M_{{\ell}_k} > k$. Next, we can choose positive integers
$j_1 < j_2 < j_3 < \dots$ such that for each $k \in \posint$
\[
  \|F^{\mu_{j_k}-{\ell}_k}(o)-\eta\| < 1/k \quad\text{and}
  \quad \koba_{\OM}\big( F^{\mu_{j_k}-\ell_k}(o),o \big) > k.
\]
These inequalities imply that $F^{\mu_{j_k}-\ell_k}(o) \to \eta$ as $k \to \infty$ and
\[
  \limsup_{k \to \infty} \koba_{\OM}\big( F^{\mu_{j_k}-\ell_k}(o),o \big) = \infty,
\]
which contradicts Proposition~\ref{P:iteration_seqs}, since $\eta \neq \xi$.
Hence $\limsup_{\ell \to \infty} M_{\ell} < +\infty$, as claimed. Then
\begin{align*}
  \limsup_{i \to \infty} \limsup_{j \to \infty} \koba_{\OM}\big( F^{\nu_i}(o),F^{\mu_j}(o) \big) &\leqslant
  \limsup_{i \to \infty} \limsup_{j \to \infty} \koba_{\OM}\big( o,F^{\mu_j-\nu_i}(o) \big) \\
  &= \limsup_{i \to \infty} M_{\nu_i} < \infty.
\end{align*}
This contradicts \eqref{E:plus_infty_outcome}, and finishes the consideration of Case~1.
 
\smallskip

\noindent{\textbf{Case~2.}} We now consider the case in which for some (and hence any) $o \in \OM$,
\[
  \limsup_{\nu \to \infty} \koba_{\OM}(F^{\nu}(o),o) < \infty.  
\]
The argument that follows is almost identical to that under the heading ``Case~2'' in the proof of
\cite[Theorem~1.10]{Bharali_Zimmer}. However, since the argument is rather technical, we reproduce
it below instead of directing the reader elsewhere. 
Recall that, by assumption, there exist two distinct points $\xi, \eta \in \Gamma$.
We choose strictly increasing sequences 
$(\nu_i)_{i \geqslant 1}, (\mu_j)_{j \geqslant 1}\subset \posint$ such that
$F^{\nu_i}\to \xi$ and $F^{\mu_j}\to \eta$ uniformly on compact subsets of $\OM$. Choose
$\overline{\OM}$-open neighbourhoods $V_{\xi}$ and $V_{\eta}$ of $\xi$ and $\eta$, respectively, such
that $\overline{V_{\xi}} \cap \overline{V_{\eta}} = \varnothing$. By the fact that $\OM$ is a visibility domain with respect
to the Kobayashi distance, there exists a compact subset $K$ of $\OM$ such that for every $(1,1)$-almost-geodesic
$\sigma: [0,T] \to \OM$ satisfying $\sigma(0) \in V_{\xi}$ and $\sigma(T) \in V_{\eta}$,
$\range(\sigma) \cap K \neq \varnothing$.   	

\smallskip

Next, for $\delta > 0$ arbitrary, we define $G_{\delta} : K \times K \to [0,+\infty)$ by
\[
  G_{\delta}(x_1,x_2) \defeq \inf\{ \koba_{\OM}( F^m(x_1),x_2 ) 
  \mid m \in \posint \text{ and } \|\xi-F^m(x_1)\| < \delta \}.
\]
By the hypothesis of Case~2,
$\sup_{\,\delta\,>\,0;~x_1,\,x_2 \in K} G_{\delta}(x_1,x_2) < \infty$.
Fix $x_1,x_2 \in K$; if
$0 < \delta_1 < \delta_2$, then $G_{\delta_1}(x_1,x_2) \geqslant G_{\delta_2}(x_1,x_2)$. So, for any
$x_1,x_2 \in K$,
\[
  G(x_1,x_2) \defeq \lim_{\delta\to 0^+}G_{\delta}(x_1,x_2)
\]
is well-defined. We also define
\[
  \epsilon \defeq \liminf_{z \to \eta} \inf_{y \in K} \koba_{\OM}(y, z).
\]
Note that by Result~\ref{R:koba_facts_misc}, $\epsilon > 0$. Choose
points $q_1,q_2 \in K$ such that
\[
  G(q_1,q_2) < \inf\{ G(x_1,x_2) \mid x_1,x_2 \in K \} + \epsilon.
\]
By an argument similar to the one leading to \eqref{E:minus-l_seq}, for each fixed $j\in \posint$
\[
  (F^{\nu_i+\mu_j})_{i \geqslant 1} \text{ converges uniformly on compact subsets to $\xi$}.
\]
Therefore, we can find a strictly increasing sequence $(i_j)_{j \geqslant 1}\subset \posint$ such that
\begin{enumerate}[leftmargin=22pt]
	\item[$(a)$] $(F^{\nu_{i_j}})_{j \geqslant 1}$ converges uniformly on compact subsets to $\xi$;
	\item[$(b)$] $(F^{\nu_{i_j}+\mu_j})_{j \geqslant 1}$ converges uniformly on compact subsets
	to $\xi$; and
	\item[$(c)$] $\lim_{j \to \infty} \koba_{\OM}(F^{\nu_{i_j}}(q_1),q_2) = G(q_1,q_2)$. 
\end{enumerate}
Finally, choose a sequence $(\kappa_j)_{j \geqslant 1}$ such that $0 < \kappa_j \leqslant 1$ for all $j$ and such that
$\kappa_j \to 0^+$ as $j \to \infty$. By 
Proposition~4.4 of \cite{Bharali_Zimmer}\,---\,which guarantees the existence of $(\lambda, \kappa)$-almost-geodesics
joining a given pair of points, for any $\lambda\geqslant 1$ and $\kappa > 0$\,---\,for each $j$
there exists a $(1,\kappa_j)$-almost-geodesic $\sigma_j : [0,T_j] \to \OM$ such that $\sigma_j(0) =
F^{\nu_{i_j}+\mu_j}(q_1)$ and $\sigma_j(T_j) = F^{\mu_j}(q_2)$. Clearly, for sufficiently large $j$, $\sigma_j(0) \in
V_{\xi}$ and $\sigma_j(T_j) \in V_{\eta}$. Now, because every $\sigma_j$ is a $(1,1)$-almost-geodesic, 
$\range(\sigma_j) \cap K \neq \varnothing$ for each $j$. Hence, for each $j$, we may choose a point $x^{*}_j \in
\range(\sigma_j) \cap K$. Since $K$ is compact, we may, by passing to a subsequence and relabelling, assume
that $x^{*}_j \to x^0 \in K$ as $j \to \infty$. Therefore, by
Lemma~\ref{L:basic_but_important}, 
we have
\begin{equation} \label{E:kob_dist_ineq_geod_cpt} 
  \koba_{\OM}(F^{\nu_{i_j}+\mu_j}(q_1),F^{\mu_j}(q_2)) \geqslant \koba_{\OM}(F^{\nu_{i_j}+\mu_j}(q_1),x^{*}_j) +
  \koba_{\OM}(x^{*}_j, F^{\mu_j}(q_2)) - 3\kappa_j.
\end{equation}
Now
\begin{align} 
  \liminf_{j \to \infty} \koba_{\OM}(F^{\nu_{i_j}+\mu_j}(q_1),x^{*}_j)
  &\geqslant \liminf_{j \to \infty} \big( \koba_{\OM}(
  F^{\nu_{i_j}+\mu_j}(q_1), x^0)-\koba_{\OM}(x^0, x^{*}_j) \big) \notag \\
  &= \liminf_{j \to \infty} \koba_{\OM}(F^{\nu_{i_j}+\mu_j}(q_1), x^0) 
       - \lim_{j \to \infty} \koba_{\OM}(x^0 ,x^{*}_j) \notag \\
  &= \liminf_{j \to \infty} \koba_{\OM}(F^{\nu_{i_j}+\mu_j}(q_1), x^0) \geqslant G(q_1, x^0). 
  \label{E:K-sep_sub-sub-sequence}
\end{align}
Again, from the definition of $\epsilon$,
\[
  \liminf_{j \to \infty} \koba_{\OM}(x^{*}_j,F^{\mu_j}(q_2)) \geqslant \epsilon.
\]
Therefore from \eqref{E:kob_dist_ineq_geod_cpt} we obtain, since $\lim_{j \to \infty} \kappa_j = 0$:
\begin{align*}
   \liminf_{j \to \infty} \koba_{\OM}(F^{\nu_{i_j}+\mu_j}(q_1),F^{\mu_j}(q_2)) &\geqslant \liminf_{j \to \infty}
   \koba_{\OM}(F^{\nu_{i_j}+\mu_j}(q_1),x^{*}_j) \\
   &\qquad\quad + \liminf_{j \to \infty} \koba_{\OM}(x^{*}_j,F^{\mu_j}(q_2)) \\
   &\geqslant G(q_1, x^0)+\epsilon.  && (\text{by \eqref{E:K-sep_sub-sub-sequence} above})
\end{align*}
On the other hand,
\[
  \limsup_{j \to \infty} \koba_{\OM}(F^{\nu_{i_j}+\mu_j}(q_1),F^{\mu_j}(q_2)) \leqslant \limsup_{j \to \infty}
  \koba_{\OM}(F^{\nu_{i_j}}(q_1),q_2) 
  = G(q_1,q_2).   
\]
Recall that the sequence $(i_j)_{j \geqslant 1}$ has been so picked that the last equality
holds true. From the last two inequalities we obtain
\[
  G(q_1,q_2) \geqslant G(q_1,x^0) + \epsilon,
\]
which is a contradiction to the choice of $q_1$ and $q_2$. This finishes the consideration of Case~2.

\smallskip

The above arguments show that we have a contradiction in either case, whence our assumption about $\Gamma$
must be wrong. This completes the proof of the theorem. 
\end{proof}

The conclusions of the last theorem constitute a step in the following:

\begin{proof}[The proof of Theorem~\ref{T:Wolff_Denjoy_special}]
Since $\OM$ is a taut bounded domain, Result~\ref{R:koba_facts_miscGeom}-(\ref{Cnclsn:taut_domains_pcvx}) tells us that
it is pseudoconvex.
It is well known\,---\,see Theorem~4.2.7 of \cite{Hormander}, for instance\,---\,that
$H^j(\OM;\C)=0$ for all $j > n$ since $\OM$ is pseudoconvex. By the universal coefficient theorem,
\[
\dim_{\C}(H^j(\OM; \C)) = \dim_{\mathbb{Q}}(H^j(\OM; \mathbb{Q})) = \textrm{rank}(H_j(\OM; \Z))
\quad\forall j\in \N.
\]
From the last two statements, together with our hypothesis, it follows that
\begin{align*}
  H^j(\OM;\mathbb{Q}) &= 0 \quad \forall j\in \N, \text{ $j$ odd}, \\
  \dim_{\mathbb{Q}}(H^j(\OM; \mathbb{Q})) &< \infty \quad \forall j\in \N, \text{ $j$ even}.
\end{align*}
Therefore we can invoke Corollary~2.10 from the article \cite{Abate_ItThCptDivSeq} by Abate 
to conclude that either $F$ has a
periodic point in $\OM$ or $(F^{\nu})_{\nu \geqslant 1}$ is compactly divergent. In the latter case, the first
outcome of the dichotomy presented by Theorem~\ref{T:Wolff_Denjoy_gen} cannot hold true. Thus,
by Theorem~\ref{T:Wolff_Denjoy_gen}, there exists a $\xi \in \bdy \OM$ such that
$(F^{\nu})_{\nu \geqslant 1}$ converges uniformly on compact subsets of $\OM$ to $\textsf{const}_{\xi}$.
However, this conclusion is not possible if $F$ has a periodic
point in $\OM$. So, in this case, the dichotomy presented by
Theorem~\ref{T:Wolff_Denjoy_gen} implies that for each $z \in \OM$, the orbit
$\{F^{\nu}(z) \mid \nu \in \posint\}$
is relatively compact in $\OM$. This completes the proof of the theorem.
\end{proof}

We are finally in a position to give a proof of Corollary~\ref{C:Wolff_Denjoy_example}.
The phrase ``finite type'' refers to the finiteness of the D'Angelo
$1$-type. We shall not define this term here: we refer the reader to 
\cite{DAngelo} for a definition.

\begin{proof}[The proof of Corollary~\ref{C:Wolff_Denjoy_example}]
If $\OM$ is a bounded pseudoconvex domain of finite type, then, by definition, $\bdy\OM$ is
at least $\smoo^2$-smooth. Therefore, it follows from a theorem of Kerzman--Rosay
\cite[Proposition~2]{Kerzman_Rosay} that $\OM$ is taut. On the other hand, if
$\OM$ is a caltrop, then we have shown\,---\,see Theorem~\ref{T:caltrops_taut}
above\,---\,that $\OM$ is taut.
\smallskip

If $\OM$ is a bounded pseudoconvex domain of finite type, then\,---\,in the terminology
of \cite{Bharali_Zimmer}\,---\,it is a Goldilocks domain and thus a visibility domain;
see Section~2 and Theorem~1.4 of 
\cite{Bharali_Zimmer}. If $\OM$ is a caltrop, then Theorem~\ref{T:visibility-caltrops}
tells us that it is a visibility domain.
Thus, in either case, $\OM$ satisfies all the conditions of
Theorem~\ref{T:Wolff_Denjoy_special}. Hence, the corollary follows.
\end{proof}

\section*{Acknowledgments}\vspace{-1mm}
\noindent{Gautam Bharali is supported by a Swarnajayanti Fellowship (Grant No.~DST/SJF/MSA-02/2013-14).
Anwoy Maitra is supported by a scholarship from the Indian Institute of Science. Both authors are also supported
by a UGC CAS-II grant (Grant No. F.510/25/CAS-II/2018(SAP-I)).}


\begin{thebibliography}{99}
\vspace{-1mm}

\bibitem{Abate88}
Marco Abate, \emph{Horospheres and iterates of holomorphic maps},
Math. Z. \textbf{198} (1988), no.~2, 225--238
	
\bibitem{Abate_ItThTautMan}
Marco Abate, Iteration Theory of Holomorphic Maps on Taut Manifolds, Research and Lecture Notes in Mathematics, Complex
Analysis and Geometry, Mediterranean Press, Rende 1989.

\bibitem{Abate_ItThCptDivSeq}
Marco Abate, \emph{Iteration theory, compactly divergent sequences and commuting holomorphic maps}, Ann. Scuola
Norm. Sup. Pisa Cl. Sci. (4) \textbf{18} (1991), no.~2, 167--191.

\bibitem{Abate_Raissy}
Marco Abate and Jasmin Raissy, \emph{Wolff--Denjoy theorems in nonsmooth convex domains},
Ann. Mat. Pura Appl. \textbf{193} (2014), no.~5, 1503--1518.	

\bibitem{Ballmann_Gromov_Schroeder}
Werner Ballmann, Mikhael Gromov, and Viktor Schroeder, 
Manifolds of Nonpositive Curvature, Progress in Mathematics \textbf{61},
Birkh{\"a}user, Boston, MA, 1985.
	
\bibitem{Balogh_Bonk}
Zolt{\'a}n M.~Balogh and Mario Bonk, \emph{Gromov hyperbolicity and the Kobayashi metric on strictly
pseudoconvex domains}, Comment. Math. Helv. \textbf{75} (2000), no.~3, 504--533.

\bibitem{Beardon1990}
A.F.~Beardon, \emph{Iteration of contractions and analytic maps},
J. London Math. Soc. \textbf{41} (1990), no.~1, 141--150. 

\bibitem{Beardon97}
A.F.~Beardon, \emph{The dynamics of contractions}, Ergodic Theory Dynam. Systems
\textbf{17} (1997), no.~6, 1257--1266.

\bibitem{Bharali_Zimmer}
Gautam Bharali and Andrew Zimmer, \emph{Goldilocks domains, a weak notion of visibility, and applications}, Adv. Math.
\textbf{310} (2017), 377--425.

\bibitem{Bridson_Haefliger}
Martin R.~Bridson and Andr{\'e} Haefliger,
Metric Spaces of Non-positive Curvature, Grundlehren der Mathematischen Wissenschaften
\textbf{319}, Springer-Verlag, Berlin, 1999.

\bibitem{Budzynska}
Monika Budzy{\'n}ska, \emph{The Denjoy--Wolff theorem in $\C^n$}, Nonlinear Anal. \textbf{75}
(2012), no.~1, 22--29.

\bibitem{DAngelo}
John P.~D'Angelo, \emph{Real hypersurfaces, orders of contact, and applications},
Ann. of Math. \textbf{115} (1982), no.~3, 615--637.

\bibitem{Denjoy}
A.~Denjoy, \emph{Sur l'it{\'e}ration des fonctions analytiques}, C.R. Acad. Sci. Paris \textbf{182}
(1926), 255--257.

\bibitem{Eberlein_ONeill}
P.~Eberlein and B.~O'Neill, \emph{Visibility manifolds}, Pacific J. Math. \textbf{46} (1973) 45--109.

\bibitem{Forstneric_Rosay}
Franc Forstneric and J.-P.~Rosay, \emph{Localization of the Kobayashi metric and the boundary continuity of
proper holomorphic mappings}, Math. Ann. \textbf{279} (1987), 239--252.

\bibitem{Herve}
Michel Herv{\'e}, \emph{Quelques propri{\'e}t{\'e}s des applications analytiques d'une boule {\`a}
 $m$ dimensions dan elle-m{\^e}me}, J. Math. Pures Appl. \textbf{42} (1963), 117--147.

\bibitem{Hormander}
Lars Hormander, An Introduction to Complex Analysis in Several Variables, North-Holland Mathematical Library,
North-Holland, Elsevier Science Publishers B.V., Amsterdam 1990.
 
\bibitem{Huang}
Xiao Jun Huang, \emph{A non-degeneracy property of extremal mappings and iterates of
holomorphic self-mappings}, Ann. Scuola Norm. Sup. Pisa Cl. Sci. \textbf{21} (1994),
no.~3, 399--419.

\bibitem{Jarnicki_Pflug}
Marek Jarnicki and Peter Pflug, Invariant Distances and Metrics in Complex Analysis, volume~9,
de Gruyter Expositions in Mathematics, Walter de Gruyter \& Co., Berlin, 1993.

\bibitem{Karlsson}
Anders Karlsson, \emph{Non-expanding maps and Busemann functions}, Ergodic Theory Dynam. Systems
\textbf{21} (2001), no.~5, 1447-1457.

\bibitem{Kerzman_Rosay}
Norberto Kerzman and Jean-Pierre Rosay,
\emph{Fonctions plurisousharmoniques d'exhaustion born{\'e}es et domaines
taut}, Math. Ann. \textbf{257} (1981), no.~2, 171--184.

\bibitem{Kiernan}
Peter Kiernan, \emph{On the relations between taut, tight and hyperbolic
  manifolds}, Bull. Amer. Math. Soc. \textbf{76} (1970), 49--51.
  
\bibitem{Ma}
Daowei Ma, \emph{Sharp estimates of the Kobayashi metric near strongly pseudoconvex points}, in
The Madison Symposium on Complex Analysis (Madison, WI, 1991), pp.~329--338,
Contemp. Math. \textbf{137} Amer. Math. Soc., Providence, RI, 1992.

\bibitem{Reich_Shoikhet}
Simeon Reich and David Shoikhet, \emph{Denjoy-Wolff theorem}, Encyclopedia of Mathematics,
URL:~\texttt{http://encyclopediaofmath.org/index.php?title=Denjoy-Wolff$\underline{\hspace{0.2cm}}$theorem\&oldid=49926}.

\bibitem{Royden}
H.~L. Royden, \emph{Remarks on the {K}obayashi metric},
in Several Complex Variables, II (Proc. Internat. Conf.,
Univ. Maryland, College Park, MD, 1970), pp.~125--137, Lecture
Notes in Math. \textbf{185}, Springer, Berlin, 1971.

\bibitem{Sibony}
Nessim Sibony, \emph{A class of hyperbolic manifolds}, in Recent Developments in Several Complex Variables
(Proc. Conf., Princeton Univ., Princeton, NJ, 1979), pp.~357--372, Ann. of Math. Stud. \textbf{100}
Princeton Univ. Press, Princeton, NJ, 1981.

\bibitem{Wolff}
J.~Wolff, \emph{Sur une g{\'e}n{\'e}ralisation d'un th{\'e}or{\`e}me de Schwarz},
C.R. Acad. Sci. Paris \textbf{182} (1926), 918--920.

\bibitem{Wu}
H.~Wu, \emph{Normal families of holomorphic mappings}, Acta Math. \textbf{119} (1967), 193--233.

\bibitem{Zimmer2016}
Andrew~M. Zimmer,
\emph{Gromov hyperbolicity and the Kobayashi metric on convex domains of finite type},
Math. Ann. \textbf{365} (2016), no.~3-4, 1425--1498.

\bibitem{Zimmer2017}
Andrew M.~Zimmer, \emph{Characterizing domains by the limit set of their automorphism group},
Adv. Math. \textbf{308} (2017), 438--482.
\end{thebibliography}
\end{document}